\documentclass[times,doublespace,11pt,oneside,reqno]{amsart}

\usepackage{graphicx}
\usepackage{amssymb}
\usepackage{epstopdf}

\usepackage[a4paper, total={6in, 9in}]{geometry}

\usepackage{amsmath,amsfonts,amsthm,mathrsfs,amssymb,cite}
\usepackage[usenames]{color}

\usepackage{hyperref}
\usepackage{xcolor}
\hypersetup{
  colorlinks,
  linkcolor={blue!80!black},
  urlcolor={blue!80!black},
  citecolor={blue!80!black}
}
\usepackage[capitalize,noabbrev]{cleveref}

\newtheorem{thm}{Theorem}[section]

\newtheorem{lem}{Lemma}[section]

\theoremstyle{definition}
\newtheorem{defn}{Definition}[section]

\theoremstyle{remark}

\newtheorem{rem}{Remark}[section]
\numberwithin{equation}{section}

\numberwithin{equation}{section}

\newcounter{saveeqn}

\newcommand{\eqnref}[1]{(\ref {#1})}

\newcommand{\Bn}{\mathbf{n}}

\newcommand{\Bz}{\mathbf{z}}

\newcommand{\Bx}{\mathbf{x}}
\newcommand{\By}{\mathbf{y}}

\newcommand{\Gl}{\lambda}

\newcommand{\Gs}{\sigma}

\newcommand{\Kcal}{\mathcal{K}}

\newcommand{\Scal}{\mathcal{S}}

\newcommand{\Ocal}{\mathcal{O}}

\newcommand{\ds}{\displaystyle}
\newcommand{\la}{\langle}
\newcommand{\ra}{\rangle}

\newcommand{\RR}{\mathbb{R}}

\newcommand{\p}{\partial}

\newcommand{\beq}{\begin{equation}}
\newcommand{\eeq}{\end{equation}}

\title[Plasmon resonances of nanorods]{Plasmon resonances of nanorods in transverse electromagnetic scattering}

\date{} 
\begin{document}
\maketitle

\begin{center}
\author{Youjun Deng\footnote{School of Mathematics and Statistics, HNP-LAMA, Central South University, Changsha, Hunan, China. Email: youjundeng@csu.edu.cn; dengyijun\_001@163.com}\,,\ Hongyu Liu\footnote{Department of Mathematics, City University of Hong Kong, Kowloon, Hong Kong SAR, China. Email: hongyu.liuip@gmail.com; hongyliu@cityu.edu.hk}\ \, and\ Guang-Hui Zheng\footnote{School of Mathematics, Hunan Provincial Key Laboratory of Intelligent Information Processing and Applied Mathematics, Hunan University, Changsha, Hunan, China. Email: zhenggh2012@hnu.edu.cn}}
\end{center}

\begin{abstract}
Plasmon resonance is the resonant oscillation of conduction electrons at the interface between negative and positive permittivity material stimulated by incident light, which forms the fundamental basis of many cutting-edge industrial applications. We are concerned with the quantitative theoretical understanding of this peculiar resonance phenomenon. It is known that the occurrence of plasmon resonance as well as its quantitative behaviours critically depend on the geometry of the material structure, the corresponding material parameters and the operating wave frequency, which are delicately coupled together. In this paper, we study the plasmon resonance for a 2D nanorod  structure, which presents an anisotropic geometry and arises in the transverse electromagnetic scattering. We present delicate spectral and asymptotic analysis to establish the accurate resonant conditions as well as sharply characterize the quantitative behaviours of the resonant field.




\end{abstract}

\section{Introduction}
\subsection{Mathematical setup and summary of major findings}

Initially focusing on the mathematics, but not the physics, we present the mathematical setup as well as summarize the major findings of our study.

Consider the following Helmholtz system in $\RR^2$:
\beq\label{eq:helm01}
\left\{
\begin{split}
&\nabla\cdot \Big(\frac{1}{\varepsilon(\Bx)} \nabla u (\Bx)\Big) + \omega^2\mu(\Bx) u(\Bx)= 0, \quad \quad \, \, \Bx \in \RR^2,\\
&\ \ u(\Bx)=u^i(\Bx)+u^s(\Bx),\hspace*{3.2cm} \Bx\in\mathbb{R}^2,\\
&\lim_{|\Bx|\rightarrow \infty}|\Bx|^{1/2}\Big(\frac{\Bx}{|\Bx|}\cdot\nabla u^s(\Bx)-\mathrm{i} \omega u^s (\Bx)\Big)=0,
\end{split}
\right.
\eeq
where $\mathrm{i}:=\sqrt{-1}$, $\omega\in\mathbb{R}_+$ signifies the temporal frequency and the last limit holds uniformly in $\hat\Bx:=\Bx/|\Bx|\in\mathbb{S}^1$, $\Bx=(x_1, x_2)\in\mathbb{R}^2$. The PDE system \eqref{eq:helm01} describes the transverse electromagnetic scattering \cite{LL15}. Here, $\varepsilon$ and $\mu$ are respectively the (relative) electric permittivity and magnetic permeability, which characterize the medium configuration of the space. By a standard normalization, throughout the rest of the paper, we assume that $\mu\equiv 1$ and moreover $(\varepsilon(\Bx)-1)$ is compactly supported which shall be fixed shortly. In \eqref{eq:helm01}, $u^i$ signifies an incident wave field satisfying $\Delta u^i +\omega^2 u^i=0$ in $\RR^2$ and, $u^s$ and $u$ are respectively referred to as the scattered and total wave fields.

Suppose $\varepsilon(\Bx)=(\varepsilon_c-1)\chi(D_\delta)+1$ with $\varepsilon_c\in\mathbb{C}$, where $\chi(D_\delta)$ is the indicator/characteristic function for $D_\delta$.  $\varepsilon_c$ is the material parameter of $D_\delta$, which is a key and subtle ingredient in our resonance study and shall be delicately determined in what follows.  We consider that $D_\delta$ is a nanorod shape. To define such a shape, let $\Gamma_0$ be a straight line with the parametrization $\Gamma_0(t)$, $t\in (t_0, t_1)$. Here we set $t_0=-L/2$ and $t_1=L/2$, where $L$ is a positive constant. Define $\Gamma_0(t_0):=P$, $\Gamma_0(t_1)=Q$, and let $\Bn(t)$ be the normal direction of $\Gamma_0(t)$, respectively.
The nanorod $D_\delta$ is defined by $D_\delta=\overline{D_\delta^a}\cup D_\delta^f \cup \overline{D_\delta^b}$, where $D_\delta^f$ is defined by
\beq
D_\delta^f:=\{\Bx(t)| \Bx(t)=\Gamma_0(t)\pm \delta \Bn(t), \quad t\in (t_0, t_1)\}.
\eeq
The two caps $D_\delta^a$ and $D_\delta^b$ are 
two half disks with radius $\delta$ and centering at $P$ and $Q$, respectively.
Let $S_\delta^a$ and $S_\delta^b$ be the surface of $D_\delta^a$ and $D_\delta^b$, respectively.
It can be verified that $D_\delta$ is of class $C^{1,\alpha}$ for a certain $\alpha\in\mathbb{R}_+$ which depends on $\delta$. In what follows, we define $S_\delta^c:=\p D_\delta^c=\p (D_\delta^a\cup D_\delta^b)$, and $S_\delta^f:=\p D_\delta^f$. As we are considering the straight nanorod, $S_\delta^f=\Gamma_{1,\delta}\cup \Gamma_{2,\delta}$, where $\Gamma_{j,\delta}$, $j=1, 2$ are defined by
\beq
\Gamma_{1,\delta}=\{\Bx;\ \Bx=\Gamma_0-\delta \Bn\}, \quad \Gamma_{2,\delta}=\{\Bx;\ \Bx=\Gamma_0+\delta \Bn\}.
\eeq
In particular, when $\delta=1$, we also set $\partial D:=\partial D_1$, $S^{l}:=S^l$ ($l=a, b, c, f$), $\Gamma_{j}:=\Gamma_{j,1}$ ($j=1, 2$) for simplicity. Moreover, we shall always use $\Bz_x$ and $\Bz_y$ to signify the projections of $\Bx\in S_\delta^f$ and $\By\in S_\delta^f$ on $\Gamma_0$, respectively.

We shall consider our study in the quasi-static regime, which signifies that size of $D_\delta$ is much smaller than the operating wavelength $2\pi/\omega$. Noting that in our study, we require that $L\sim 1$ and $\delta\ll 1$, and hence in general we shall always assume that $\omega\ll 1$. Nevertheless, at this point, we would like to emphasize that the two asymptotic parameters $\omega$ and $\delta$ are also delicately related, which shall be observed later.


With the above preparation, the first main result of our study is the following asymptotic representation of the scattered wave field $u^s$ from a nanorod material structure.

\begin{thm}\label{eq:thmanin01}
Let $u$ be the solution to \eqnref{eq:helm01}, where $D_\delta$ is the nanorod described above. Let the incident wave $u^i$ be the plane wave, i.e. $u^i=e^{i\omega \mathbf{d}\cdot \Bx}$, $\mathbf{d}=(d_1, d_2)\in\mathbb{S}^1$. Then for $\Bx\in \RR^2\setminus\overline{D_\delta}$, it holds that
\beq\label{eq:thmanin0101}
\begin{split}
u^s(\Bx)=&\omega\delta\frac{\mathrm{i}}{2\pi}d_2\int_{-L/2}^{L/2}\frac{x_2}{(x_1-y_1)^2+x_2^2}(\lambda(\varepsilon_c) I+A_\delta)^{-1}[1](y_1)dy_1\medskip\\
&+\omega\delta\frac{\mathrm{i}}{2\pi}\Big(\lambda(\varepsilon_c)-\frac{1}{2}\Big)^{-1} d_1\ln\frac{(x_1+L/2)^2+x_2^2}{(x_1-L/2)^2+x_2^2} +\omega\cdot o(\delta)+\mathcal{O}(\omega^2\ln\omega),
\end{split}
\eeq
where
$$\lambda(\varepsilon_c)=\frac{1}{2}\cdot\frac{1+\varepsilon_c}{1-\varepsilon_c},$$
and the operator $A_\delta$ shall be defined in \eqnref{eq:defAop01}.
\end{thm}
\begin{rem}
The formula \eqref{eq:thmanin0101} presents a neat and concise representation of the wave field from the scattering of a nanorod material structure, which is of significant practical interest for its own sake; see e.g. \cite{fang2021} for the related discussion in a different physical context. It is interesting to see from \eqnref{eq:thmanin0101} that, if the incident wave is propagating parallel to the nanorod, i.e., $d_2=0$, then the formula yields:
\beq\label{eq:aa1}
u^s(\Bx)\sim\omega\delta\frac{\mathrm{i}}{2\pi}\Big(\lambda(\varepsilon_c)-\frac{1}{2}\Big)^{-1} d_1\ln\frac{(x_1+L/2)^2+x_2^2}{(x_1-L/2)^2+x_2^2}.
\eeq
From \eqref{eq:aa1}, it is readily observed that the scattered field becomes much stronger if one approaches the two ends of the nanorod $P_0$ and $Q_0$.
\end{rem}

Theorem~\ref{eq:thmanin01} paves the way for our study of the plasmon resonance associated with the nanorod material structure $(\varepsilon_c, D_\delta)$. Before that, we first present the notion of plasmon resonance mathematically.
\begin{defn}\label{depr}
Consider the Helmholtz system \eqnref{eq:helm01} associated with the nanorod $(D_\delta, \varepsilon_c)$.
Then plasmon resonance occurs if the following condition is fulfilled:
\begin{align}\label{prdf01}
\left\|\nabla u^s\right\|_{L^2(\mathbb{R}^2\setminus\overline{D_\delta})}\gg 1.
\end{align}
\end{defn}
We mention that there are different definitions of plasmon resonance but with essentially similar formulation, see e.g. \cite{AMRZ14,DLU72}. The energy blowup is a hallmark feature of the plasmon resonance which implies that the resonant field exhibits highly oscillatory patterns. In the occurrence of plasmon resonance, the highly oscillating behaviour of the resonant field happens near the boundary of the nanostructure, and hence plasmon resonance is also referred to as the Surface Localized Resonance (SLR) or Surface Plasmon Resonance (SPR).

It is clear that if $\varepsilon_c\in\mathbb{C}$ is a regular permittivity with $\Re\varepsilon\in\mathbb{R}_+$ and $\Im\varepsilon\in\mathbb{R}_+$, then by the well-posedness of the Helmholtz system \eqref{eq:helm01} (cf. \cite{LSSZ}), the plasmon resonance cannot occur. Hence, in order to induce the plasmon resonance, one actually has $\Re\varepsilon_c\leq 0$ and $\Im\varepsilon_c>0$, and moreover they are delicately coupled with the nanorod geometry $D_\delta$ and the operating frequency $\omega$ through the spectrum of a certain integral operator. This shall become clearer in our subsequent analysis. It is pointed out that materials with those ``non-natural" parameters are referred to as metameterials in the literature. In order to provide a global picture of our study, the main resonance result in our study can be briefly summarized in the following theorem.
\begin{thm}\label{thm0101}
Let $u^s$ be the scattering solution of (\ref{eq:helm01}). Assume that
\begin{equation}
\omega^2\ln\omega|\rho|^{-1}\leq c_1,
\end{equation}
where
$$\rho:=\Im\left(\frac{1}{\varepsilon_c}\right)<0,$$
for a sufficiently small $c_1$, and moreover $\Re\varepsilon_c\leq 0$ is properly given. If the incident wave $u^i$ is well chosen and the parameter $\rho$ fulfils that $|\rho|^{-1}\omega\delta\rightarrow\infty$ (as $\omega\rightarrow 0$, $\delta\rightarrow 0$, and $|\rho|\rightarrow 0$), then it holds
\begin{align}
\left\|\nabla u^s\right\|_{L^2(\mathbb{R}^2\setminus\overline{D_\delta})}\rightarrow\infty.
\end{align}
\end{thm}
\begin{rem}
Two remarks are in order. First, in Theorem \ref{thm0101}, we mention that both $\Re\varepsilon_c$ and $u^i$ should be properly given. Indeed, they are delicately and subtly connected to the spectral properties of certain integral operators. The technical details will be given in Theorem \ref{thm4.3} in what follows.
Second, the condition $\omega^2\ln\omega|\rho|^{-1}\leq c_1$ and  $|\rho|^{-1}\omega\delta\rightarrow\infty$ implies that $\delta^{-1}(\omega\ln\omega)\rightarrow0$ as $\omega\rightarrow0$ and $\delta\rightarrow0$. In such a case, by the theorem, the gradient of the scattered field blows up. According to \eqref{prdf01} in Definition \ref{depr}, we see that the plasmon resonance occurs. Notice that the last two conditions on the resonant material configuration are very flexible, and for example, one can take $|\rho|=\Ocal(\omega^{\frac{3}{2}})$, $\delta=\Ocal(\omega^{\frac{1}{3}})$.
\end{rem}

\subsection{Background and discussion}
Plasmon resonance is the resonant oscillation of conduction electrons at the interface between negative and positive permittivity material stimulated by incident light, which forms the fundamental basis of many cutting-edge applications. The plasmonic technology is revolutionizing many industrial applications including enhancing the brightness of light, confining strong electromagnetic fields, medical therapy, invisibility cloaking and biomedical imaging; see e.g.\cite{ACKLM5,ADM6,BGQ2,BS11,BLLW,CKKL7,JLEE3,KLSW8,LL15,LLL9,LLL16,SC1,WN10,YA18,Z22} and the references cited therein.

Metallic nanostructures are often used to construct various plasmonic devices. The metallic nanoparticle exhibits nanoscale uniformly in all dimensions, a.k.a. isotropic geometry, which is the simplest nanostructure. Recently, there are extensive and intensive studies on mathematically characterizing the plasmon resonances associated with nanoparticles; see \cite{ACKLM5,ACL20,ADM6,AMRZ14,AKL13,BS11,BLLW,CKKL7,FDL15,KLSW8,LL15,LLL9,LLL16,SC1,WN10,Z22} and the references cited therein. The metallic nanorod is another important nanostructure, which has a long aspect ratio and possesses different size scales in different dimensions, a.k.a. anisotropic geometry. Metallic nanorods have been widely used in real applications including semiconductor materials, microelectromechanical systems, food packaging, catalysis, energy storage, biomedicine and cloaking \cite{LWJC41,JGM40,HDA42,VTBH44,D43}. In particular, some nanorods (such as $C_eO_2$) display higher catalytic activity compared to the nanoparticles, which could potentially increase their usage \cite{LWJC41}. However, to our best knowledge, there is little mathematical study in the literature on theoretically characterizing the plasmon resonances associated with nonorods. In \cite{RS}, the authors study the plasmon resonance associated with a certain axisymmetric slender body in the quasi-static regime via the matched asymptotics method. In \cite{DLZ21}, the authors investigate the plasmon resonance of curved nanorod for 3D Helmholtz system, which describes the acoustic scattering.

In this article, we focus on the mathematical analysis on the plasmon resonances associated with a 2D straight nanorod, which arises in the transverse electromagnetic scattering. The technical novelty of our study can be summarized as follows. Compared to the 3D study in \cite{DLZ21}, we give more delicate asymptotic and spectral analysis of the layer potential operators, especially the single layer potential operator and Neumann-Poincar\'e operator. In fact, by looking at the aspect ratios, the nanorod in 2D present more severe geometric singularities and challenges than that in 3D. This can be partly evidenced by the severe logarithmic singularity of the fundamental solution of the 2D Helmholtz equation, compared to the weaker singularity of the 3D fundamental solution. Nevertheless, through delicate and subtle asymptotic and spectral analysis, we manage to derive an accurate asymptotic formula of the scattering field, from which we can further establish the sharp conditions that ensure the occurrence of the plasmon resonance. In \cite{DLZ21}, the 3D nanorod can be curved, whereas in the current study, we mainly consider the straight nanorod. This enables to derive much more accurate understandings of the plasmon resonance as well as its quantitative relationships to the metamaterial parameters, the wave frequency and the nanorod geometry.

The rest of the paper is organized as follows. In Section 2, we present the layer potential theory and several technical asymptotic expansions. The quantitative analysis of the scattered field is derived in Section 3. Finally, we establish the mathematical analysis of the plasmon resonance for the nanorod in Section 4.

\section{Auxiliary results}

In this section, we establish several technical auxiliary results for our subsequent use.

\subsection{Layer potential operators}

Our analysis heavily relies on the layer potential technique. We briefly present some preliminary knowledge on the layer potential theory and refer to \cite{ACKLM5,ADM6,AMRZ14,AKL13,DLU7,FDL15} and the references cited therein for more related results.

Let $ G_k$ be the outgoing fundamental solution to the PDE operator $\Delta+k^2$ in $\RR^2$, which is given by:
\begin{equation}
\label{Gk} \ds G_k(\Bx) = -\frac{\mathrm{i}}{4} H_0^{(1)}(k|\Bx|),
\end{equation}
 where $H_0^{(1)}(k|\Bx|)$ is the Hankel function of first kind of order zero.
For any bounded Lipschitz domain $B\subset \RR^2$, we denote by $\Scal_B^k: H^{-1/2}(\p B)\rightarrow H^{1}(\RR^d\setminus\p B)$ the single layer potential operator given by
\beq\label{eq:layperpt1}
\Scal_{B}^k[\varphi](\Bx):=\int_{\p B}G_k(\Bx-\By)\varphi(\By)\; d\sigma(\By),
\eeq
and $(\Kcal_B^k)^*: H^{-1/2}(\p B)\rightarrow H^{-1/2}(\p B)$ the Neumann-Poincar\'e operator
\beq\label{eq:layperpt2}
(\Kcal_{B}^k)^{*}[\varphi](\Bx):=\mbox{p.v.}\quad\int_{\p B}\frac{\p G_k(\Bx-\By)}{\p \nu(\Bx)}\varphi(\By)\; d\sigma(\By),
\eeq
where p.v. stands for the Cauchy principle value. In \eqref{eq:layperpt2} and also in what follows, unless otherwise specified, $\nu$ signifies the exterior unit normal vector to the boundary of the concerned domain.
It is known that the single layer potential operator $\Scal_B^k$ is continuous across $\p B$ and satisfies the following trace formula
\beq \label{eq:trace}
\frac{\p}{\p\nu}\Scal_B^k[\varphi] \Big|_{\pm} = (\pm \frac{1}{2}I+
(\Kcal_{B}^k)^*)[\varphi] \quad \mbox{on} \quad \p B, \eeq
where $\frac{\p }{\p \nu}$ stands for the normal derivative and the subscripts $\pm$ indicate the limits from outside and inside of a given inclusion $B$, respectively. In the following, if $k=0$, we formally set $G_k$ introduced in \eqref{Gk} to be $G_0$, which has the form
$$
G_0(\Bx):=\frac{1}{2\pi} \ln |\Bx|,
$$
and the other integral operators introduced above can also be formally defined when $k=0$. In what follows, for the sake of simplicity, we also denote by $\Kcal_B$ and $\Kcal_B^*$ be the Neumann-Poincar'e operators $\Kcal_B^0$ and $(\Kcal_B^0)^*$, respectively.

\subsection{Asymptotic expansions of layer potentials}
In the part, we shall first present some asymptotic expansions of layer potential operators with respect to the size parameter $\delta\ll1$.
Note that the single layer potential operator $\mathcal {S}_B^{0}:\ H^{-\frac{1}{2}}(\partial
B)\rightarrow H^{\frac{1}{2}}(\partial B)$ is not invertible in
$\mathbb{R}^2$. We introduce a substitute of
$\mathcal {S}_B^{0}$ as follows
\begin{align}\label{}
\widetilde{\mathcal {S}}_B^{0}[\psi]=
\begin{cases}
\mathcal {S}_B^{0}[\varphi],\ \ \ \text{if}\ (\psi,\chi(\partial B))_{-\frac{1}{2},\frac{1}{2}}=0,\\
\chi(\partial B),\ \ \ \text{if}\ \psi=\varphi_{0},
\end{cases}
\end{align}
where $(\cdot,\cdot)_{-\frac{1}{2},\frac{1}{2}}$ is the duality
pairing between $\ H^{\frac{1}{2}}(\partial B)$ and $\
H^{-\frac{1}{2}}(\partial B)$. $\varphi_{0}$ is the unique eigenfunction of static Neumann-Poincar\'{e} operator $(\mathcal
{K}_B^{0})^*$ associated with eigenvalue $\frac{1}{2}$ such
that
$$(\varphi_{0},\chi(\partial B))_{-\frac{1}{2},\frac{1}{2}}=1.$$
It can be found that $\widetilde{\mathcal {S}}_B^{0}:$ $H^{-\frac{1}{2}}(\partial B)\rightarrow\
H^{\frac{1}{2}}(\partial B)$ is invertible. Thanks to the Calder{\'o}n identity:
\begin{align}\label{}
\mathcal {K}_B^{0}\widetilde{\mathcal{S}}_B^{0}=\widetilde{\mathcal {S}}_B^{0}(\mathcal{K}^{0}_B)^*,
\end{align}
from the invertibility and positivity of $-\widetilde{\mathcal {S}}_B^{0}$, we can define the inner
product
\begin{align}\label{inner01}
\langle u,v\rangle_{\mathcal {H}^*(\partial B)}=-(u,\widetilde{\mathcal
{S}}_B^{0}[v])_{-\frac{1}{2},\frac{1}{2}}.
\end{align}
Then, we have the following results
\begin{lem}\label{}
Let $B$ be of class $C^{1,\alpha}$. Then

(1) $(\mathcal {K}_B^{0})^*$ is a compact self-adjoint
operator in the Hilbert space $\mathcal {H}^*(\partial B)$ equipped
with the inner product (\ref{inner01}), which is equivalent to the
original one;

(2) Let $\left\{\lambda_{j},\varphi_{j}\right\},\ j=0,1,2,\cdots$, be the eigenvalue
and eigenfunction pair of $(\mathcal {K}_B^{0})^*$,
here $\lambda_{0}=\frac{1}{2}$. Then,
$\lambda_{j}\in(-\frac{1}{2},\frac{1}{2}]$, and
$\lambda_{j}\rightarrow0$ as $j\rightarrow\infty$;

(3) $\mathcal {H}^*(\partial B)=\mathcal {H}_0^*(\partial
B)\oplus\{c\varphi_{0}\},\ c\in\mathbb{C}$, where $\mathcal
{H}_0^*(\partial B)=\{\varphi\in\mathcal {H}^*(\partial
B):\int_{\partial B}\varphi d\sigma=0\}$;

(4) For any $\psi\in\ H^{-\frac{1}{2}}(\partial B)$, it hold:
\begin{align}\label{specexpan01}
(\mathcal{K}^{0}_B)^*[\psi]=\sum_{j=0}^\infty\lambda_{j}a_{j}^{-1}\langle\psi,\varphi_{j}\rangle_{\mathcal
{H}^*(\partial B)}\varphi_{j},
\end{align}
where
\begin{align*}
a_{j}=\langle\varphi_{j},\varphi_{j}\rangle_{\mathcal
{H}^*(\partial B)}.
\end{align*}
\end{lem}

\begin{rem}
Note that $(\widetilde{\mathcal
{S}}_B^{0})^{-1}[\chi(\partial B)]=\varphi_{0}$, and
$-\frac{1}{2}\mathcal {I}+(\mathcal
{K}_B^{0})^*=\left(-\frac{1}{2}\mathcal {I}+(\mathcal
{K}_B^{0})^*\right)\mathcal {P}_{\mathcal {H}_0^*(\partial
B)}$, where $\mathcal {P}_{\mathcal {H}_0^*(\partial
B)}$ is the orthogonal projection onto $\mathcal {H}_0^*(\partial B)$. Furthermore, it
follows that
\begin{align}\label{re01}
\left(-\frac{1}{2}\mathcal {I}+(\mathcal
{K}_B^{0})^*\right)(\widetilde{\mathcal
{S}}_B^{0})^{-1}[\chi(\partial B)]=0.
\end{align}
\end{rem}

\textcolor{black}{We also introduce the function space $\mathcal{H}(\partial B)$ which is the space $H^{\frac{1}{2}}(\partial B)$ equipped with the following inner product
\begin{align}\label{3.2}
\langle u,v\rangle_{\mathcal {H}(\partial B)}=-((\widetilde{\mathcal
{S}}_{B}^0)^{-1}[u],v)_{-\frac{1}{2},\frac{1}{2}}.
\end{align}
It can be directly verified that $\widetilde{\mathcal{S}}_ {B}^0$ is an isometry between $\mathcal {H}^*(\partial B)$ and $\mathcal{H}(\partial B)$, and $\mathcal {H}^*(\partial B)$ is the dual space of $\mathcal{H}(\partial B)$. From now on, we use $(\cdot,\cdot)$ as the standard inner product in $\mathbb{R}^2$. The inner product (\ref{inner01}) and the corresponding norm on $\partial D_\delta$ are denoted by $\langle\cdot,\cdot\rangle$ and $\|\cdot\|$ in short, respectively. $A\lesssim B$ means $A\leq CB$ for some generic positive constant $C$. $A\approx B$ means that $A\lesssim B$ and $B\lesssim A$.
}

In what follows, we always suppose that $\delta\ll 1$ and each eigenvalue of the operator Neumann-Poincar\'e operator $\mathcal {K}_{D_\delta}^*$ is simple.
We shall first present some asymptotic expansions of the Neumann-Poincar\'e operator with respect to $\delta$.
Recalling that $\p D_\delta=S_\delta^a \cup \overline{S_\delta^f} \cup S_\delta^b$, we decompose the Neumann-Poincar\'e operator into several parts accordingly. To that end, we introduce the following
boundary integral operator:
\beq\label{eq:defbnd01}
\Kcal_{\mathcal{S},\mathcal{S}'}[\varphi](\Bx):=\chi(\mathcal{S}')\frac{1}{2\pi}\int_{\mathcal{S}} \frac{( \Bx-\By, \nu_x)}{|\Bx-\By|^2}\varphi(\By)d\sigma(\By), \quad for\quad \mathcal{S}\cap \mathcal{S}'=\emptyset.
\eeq
It is obvious that $\Kcal_{\Scal,\Scal'}$ is a bounded operator from $L^2(\Scal)$ to $L^2(\Scal')$.
For the subsequent use, we also introduce the following regions:
\begin{align}\label{}
\iota_{\delta}(P):=\{\Bx;\ |P-\Bz_{x}|=\Ocal(\delta),\ \Bx\in S_\delta^f\},\\
\iota_{\delta}(Q):=\{\Bx;\ |Q-\Bz_{x}|=\Ocal(\delta),\ \Bx\in S_\delta^f\}.
\end{align}
Define $\tilde\varphi(\tilde\Bx):=\varphi(\Bx)$, where $\Bx\in S_\delta^a, S_\delta^b$ and $\tilde\Bx\in S^a, S^b$.

The asymptotic expansion of the Neumann-Poincar\'e operator is given in \cite{fang2021}.

\begin{lem}\label{lem3.2}
The Neumann-Poincar\'e operator $\Kcal_{D_\delta}^*$ admits the following asymptotic expansion:
\beq\label{eq:leasyNP01}
\Kcal_{D_\delta}^*[\varphi](\Bx)= \Kcal_0[\varphi](\Bx) +\delta\Kcal_1[\varphi](\Bx)+ \Ocal(\delta^2),
\eeq
where $\Kcal_0$ is defined by
\beq\label{eq:leasyNP02}
\begin{split}
\Kcal_0[\varphi](\Bx)& =\chi(S_\delta^a)\Big( \Kcal_{S_\delta^f, S_\delta^a}[\varphi](\Bx)+\frac{1}{4\pi}\int_{S^a} \tilde\varphi(\tilde{\By})d\tilde{\sigma}(\tilde{\By})\Big)+
\chi(S_\delta^b)\Big( \Kcal_{S_\delta^f, S_\delta^b}[\varphi](\Bx)+\frac{1}{4\pi}\int_{S^b} \tilde\varphi(\tilde{\By})d\tilde{\sigma}(\tilde{\By})\Big) \\
&+\mathcal{A}_{\Gamma_{2,\delta}, \Gamma_{1,\delta}}[\varphi]+ \mathcal{A}_{\Gamma_{1,\delta}, \Gamma_{2,\delta}}[\varphi]+ \chi(\iota_{\delta}(P))\Kcal_{S_\delta^a, S_\delta^f}[\varphi](\Bx)+\chi(\iota_{\delta}(Q))\Kcal_{S_\delta^b, S_\delta^f}[\varphi](\Bx),
\end{split}
\eeq
and
\beq\label{eq:leasyNP03}
\begin{split}
\Kcal_1[\varphi]=& \chi(S_\delta^b)\frac{\la \Bx-P, \nu_x\ra}{2\pi|\Bx-P|^2}\int_{S^a} \tilde\varphi(\tilde{\By})d\tilde{\sigma}(\tilde{\By})+\chi(S_\delta^a)\frac{\la \Bx-Q, \nu_x\ra}{2\pi|\Bx-Q|^2}\int_{S^b} \tilde\varphi(\tilde{\By})d\tilde{\sigma}(\tilde{\By})\\
&+ \chi(S_\delta^f\setminus\iota_{\delta}(P))\left(\frac{\delta}{|\Bx-P|^2}\int_{S^a}(1-\la \tilde\By- P, \nu_x\ra) \tilde\varphi(\tilde{\By})d\tilde{\sigma}(\tilde{\By})+o\Big(\frac{\delta}{|\Bx-P|^2}\Big)\right)\\
&+\chi(S_\delta^f\setminus\iota_{\delta}(Q))\left(\frac{\delta}{|\Bx-Q|^2}\int_{S^b}(1-\la \tilde\By- Q, \nu_x\ra) \tilde\varphi(\tilde{\By})d\tilde{\sigma}(\tilde{\By})+o\Big(\frac{\delta}{|\Bx-P|^2}\Big)\right).
\end{split}
\eeq
Here, the operators $\mathcal{A}_{\Gamma_{1,\delta}, \Gamma_{2,\delta}}$ and $\mathcal{A}_{\Gamma_{2,\delta}, \Gamma_{1,\delta}}$ are defined by
\beq\label{eq:leasyNP0301}
\begin{split}
\mathcal{A}_{\Gamma_{1,\delta}, \Gamma_{2,\delta}}[\varphi](\Bx)=&\frac{1}{\pi}\chi(\Gamma_{2,\delta})\int_{\Gamma_{1,\delta}}\frac{\delta}{|\Bx-\By|^2}\varphi(\By)d\sigma(\By), \\
 \mathcal{A}_{\Gamma_{2,\delta}, \Gamma_{1,\delta}}[\varphi](\Bx)=&\frac{1}{\pi}\chi(\Gamma_{1,\delta})\int_{\Gamma_{2,\delta}}\frac{\delta}{|\Bx-\By|^2}\varphi(\By)d\sigma(\By).
\end{split}
\eeq
\end{lem}
\begin{rem}
We mention that the operator $\Kcal_1$ appeared in \eqnref{eq:leasyNP01} still depends on $\delta$. In fact, if $\Bx\in \Gamma_{j,\delta}\setminus\big(\iota_{\delta}(P)\cup \iota_{\delta}(Q)\big)$, $j=1, 2$. Then the operator $\delta\Kcal_1$ admits the following asymptotic form:
\beq\label{eq:descrK101}
\delta\Kcal_1=\chi(\iota_{\delta^\epsilon}(P)\cup \iota_{\delta^\epsilon}(Q))\Ocal(\delta^{2(1-\epsilon)}) +\Ocal(\delta^2), \quad 0<\epsilon<1.
\eeq
\end{rem}

Next, we consider the asymptotic expansions for single layer potential operator $\Scal_{D_\delta}^0$ with respect to $\delta$. To begin with, we first define the following subregions
\beq
\begin{split}\label{eq:readdnew01}
\iota_{1,\delta^{\gamma}}(\tilde{x}):&=\{\tilde{y}\mid|z_{\tilde{x}}-z_{\tilde{y}}|<\delta^\gamma,\tilde{y}\in\partial D\},\\
\iota_{1,\delta^{\gamma}}^{(j)}(\tilde{x}):&=\{\tilde{y}\mid|z_{\tilde{x}}-z_{\tilde{y}}|<\delta^\gamma,\tilde{y}\in\Gamma_j\},\ \ j=1,2,\\
\iota_{1,\delta^{\gamma}}(P):&=\{\tilde{y}\mid|P-z_{\tilde{y}}|<\delta^{\gamma},\tilde{y}\in\partial D\},\\
\iota_{1,\delta^{\gamma}}(Q):&=\{\tilde{y}\mid|Q-z_{\tilde{y}}|<\delta^{\gamma},\tilde{y}\in\partial D\},\\
\iota_{1,\delta^{\gamma}}^{(j)}(P):&=\{\tilde{y}\mid|P-z_{\tilde{y}}|<\delta^{\gamma},\tilde{y}\in\Gamma_j\},\ \ j=1,2,\\
\iota_{1,\delta^{\gamma}}^{(j)}(Q):&=\{\tilde{y}\mid|Q-z_{\tilde{y}}|<\delta^{\gamma},\tilde{y}\in\Gamma_j\},\ \ j=1,2,\\
\end{split}
\eeq
and
\beq
\begin{split}
\iota_{\delta^{\gamma}}(x):&=\{y\mid|z_{x}-z_{y}|<\delta^\gamma,y\in\partial D_\delta\},\\
\iota_{\delta^{\gamma}}^{(j)}(x):&=\{y\mid|z_{x}-z_{y}|<\delta^\gamma,y\in\Gamma_{j,\delta}\},\ \ j=1,2,\\
\iota_{\delta^{\gamma}}(P):&=\{y\mid|P-z_{y}|<\delta^{\gamma},y\in\partial D_\delta\},\\
\iota_{\delta^{\gamma}}(Q):&=\{y\mid|Q-z_{y}|<\delta^{\gamma},y\in\partial D_\delta\},\\
\iota_{\delta^{\gamma}}^{(j)}(P):&=\{y\mid|P-z_{y}|<\delta^{\gamma},y\in\Gamma_{j,\delta}\},\ \ j=1,2,\\
\iota_{\delta^{\gamma}}^{(j)}(Q):&=\{y\mid|Q-z_{y}|<\delta^{\gamma},y\in\Gamma_{j,\delta}\},\ \ j=1,2,
\end{split}
\eeq
where $\gamma\in(0,1)$ is a parameter which can be selected arbitrarily.

Meanwhile, we introduce the following boundary integral operators:
\beq
\begin{split}\label{laypot02}
\mathcal{S}_{\Lambda_1,\Lambda_2}[\tilde{\varphi}](\tilde{x}):&=\frac{1}{2\pi}\chi(\Lambda_2)\int_{\Lambda_1} \ln|z_{\tilde{x}}-z_{\tilde{y}}|\tilde{\varphi}(\tilde{y}) d\tilde{\Gs}_{\tilde{y}},\\
\mathcal{S}_{\Lambda_1,\Lambda_2}^{(0,1)}[\tilde{\varphi}](\tilde{x}):&=\frac{1}{2\pi}\chi(\Lambda_2)\int_{\Lambda_1} \ln|\delta(\tilde{x}-\tilde{y})+(1-\delta)(z_{\tilde{x}}-z_{\tilde{y}})|\tilde{\varphi}(\tilde{y}) d\tilde{\Gs}_{\tilde{y}},\\
\mathcal{S}_{\Lambda_1,\Lambda_2}^{(0,2)}[\tilde{\varphi}](\tilde{x}):&=\delta\frac{1}{2\pi}\chi(\Lambda_2)\int_{\Lambda_1} \ln|\delta(\tilde{x}-\tilde{y})+(1-\delta)(z_{\tilde{x}}-z_{\tilde{y}})|\tilde{\varphi}(\tilde{y}) d\tilde{\Gs}_{\tilde{y}},\\
\Pi_{\Lambda_1,\Lambda_2}^{(0)}[\tilde{\varphi}](\tilde{x}):&=\frac{1}{2\pi}\chi(\Lambda_2)\int_{\Lambda_1} \tilde{\varphi}(\tilde{y}) d\tilde{\Gs}_{\tilde{y}},\\
\mathcal{S}_{\Lambda_1,\Lambda_2}^{(1,0)}[\tilde{\varphi}](\tilde{x}):&=\frac{1}{2\pi}\chi(\Lambda_2)\int_{\Lambda_1} \frac{(z_{\tilde{x}}-z_{\tilde{y}},\tilde{x}-z_{\tilde{x}})}{|z_{\tilde{x}}-z_{\tilde{y}}|^2}\tilde{\varphi}(\tilde{y}) d\tilde{\Gs}_{\tilde{y}},
\end{split}
\eeq
and
\beq
\begin{split}
\mathcal{S}_{\Lambda_1,\Lambda_2}^{(1,1)}[\tilde{\varphi}](\tilde{x}):&=\frac{1}{2\pi}\chi(\Lambda_2)\int_{\Lambda_1} \ln|\tilde{x}-\tilde{y}|\tilde{\varphi}(\tilde{y}) d\tilde{\Gs}_{\tilde{y}},\\
\mathcal{S}_{\Lambda_1,\Lambda_2}^{(2,0)}[\tilde{\varphi}](\tilde{x}):&=\frac{1}{2\pi}\chi(\Lambda_2)\int_{\Lambda_1} \left(\frac{1+(\tilde{x}-z_{\tilde{x}},z_{\tilde{y}}-\tilde{y})}{|z_{\tilde{x}}-z_{\tilde{y}}|^2}-
\frac{(z_{\tilde{x}}-z_{\tilde{y}},\tilde{x}-z_{\tilde{x}})^2}{|z_{\tilde{x}}-z_{\tilde{y}}|^4}\right)
\tilde{\varphi}(\tilde{y}) d\tilde{\Gs}_{\tilde{y}},\\
\mathcal{S}_{\Lambda_1,\Lambda_2}^{(2,1)}[\tilde{\varphi}](\tilde{x}):&=\frac{1}{\pi}\chi(\Lambda_2)\int_{\Lambda_1} \frac{1}{|z_{\tilde{x}}-z_{\tilde{y}}|^2}\tilde{\varphi}(\tilde{y}) d\tilde{\Gs}_{\tilde{y}},\\
\mathcal{S}_{\Lambda_1,\Lambda_2}^{(2,2)}[\tilde{\varphi}](\tilde{x}):&=\frac{1}{2\pi}\chi(\Lambda_2)\int_{\Lambda_1} \frac{(z_{\tilde{y}}-z_{\tilde{x}},\tilde{y}-z_{\tilde{y}})}{|z_{\tilde{x}}-z_{\tilde{y}}|^2}\tilde{\varphi}(\tilde{y}) d\tilde{\Gs}_{\tilde{y}},\\
\mathcal{S}_{\Lambda_1,\Lambda_2}^{(2,3)}[\tilde{\varphi}](\tilde{x}):&=\frac{1}{2\pi}\chi(\Lambda_2)\int_{\Lambda_1} \frac{(z_{\tilde{x}}-z_{\tilde{y}},\tilde{x}-\tilde{y}+z_{\tilde{y}}-z_{\tilde{x}})}{|z_{\tilde{x}}-z_{\tilde{y}}|^2}\tilde{\varphi}(\tilde{y}) d\tilde{\Gs}_{\tilde{y}},
\end{split}
\eeq
where $\Lambda_j\subset\partial D$ $(j=1,2)$ are open surfaces.

\begin{lem}\label{le:app0103}
Let $\varphi\in\mathcal {H}^*(\partial D_\delta)$ and $\tilde{\varphi}(\tilde{x})=\varphi(x)$ for $x\in\partial D_\delta$ and $\tilde{x}\in\partial D$. Then the following asymptotic results hold:

\textcolor{black}{\begin{equation}\label{singexp1}
\mathcal{S}_{D_\delta}^0[\varphi](x)= \mathcal{S}_{D,0}[\tilde{\varphi}](\tilde{x})+(\delta\ln\delta)\mathcal{S}_{D,\ln}[\tilde{\varphi}](\tilde{x})
+\delta\mathcal{S}_{D,1}[\tilde{\varphi}](\tilde{x})+\delta^2\mathcal{S}_{D,2}[\tilde{\varphi}](\tilde{x})+\mathcal{O}(\delta^3),
\end{equation}}
where
\begin{align*}\label{}
&\mathcal{S}_{D,0}[\tilde{\varphi}](\tilde{x}):=\bigg(\mathcal{S}_{S^f\setminus\iota_{1,\delta^\gamma}(P), S^a}+\mathcal{S}_{S^f\setminus\iota_{1,\delta^\gamma}(Q),S^b}
+\mathcal{S}_{S^f\setminus\iota_{1,\delta^\gamma}^{(2)}(\tilde{x}),\Gamma_1}
+\mathcal{S}_{S^f\setminus\iota_{1,\delta^\gamma}^{(1)}(\tilde{x}),\Gamma_2}
\\
&\ \ \ +\mathcal{S}_{\iota_{1,\delta^\gamma}(\tilde{x}), S^a}^{(0,1)}+\mathcal{S}_{\iota_{1,\delta^\gamma}(\tilde{x}), S^b}^{(0,1)}
+\mathcal{S}_{\iota_{1,\delta^\gamma}^{(2)}(\tilde{x}),\Gamma_1}^{(0,1)}
+\mathcal{S}_{\iota_{1,\delta^\gamma}^{(1)}(\tilde{x}),\Gamma_2}^{(0,1)}
+\mathcal{S}_{S^a,\iota_{1,\delta^\gamma}(P)}^{(0,2)}
+\mathcal{S}_{S^b,\iota_{1,\delta^\gamma}(Q)}^{(0,2)}\bigg)[\tilde{\varphi}](\tilde{x}),
\end{align*}

\begin{align*}
\mathcal{S}_{D,\ln}[\tilde{\varphi}](\tilde{x}):=\bigg(\Pi_{S^a, S^a}^{(0)}+\Pi_{S^b,S^b}^{(0)}\bigg)[\tilde{\varphi}](\tilde{x}),
\end{align*}

\begin{align*}
\mathcal{S}_{D,1}[\tilde{\varphi}](\tilde{x}):&=\bigg(\mathcal{S}_{S^a,S^f\setminus\iota_{1,\delta^\gamma}(P)}
+\mathcal{S}_{S^b,S^f\setminus\iota_{1,\delta^\gamma}(Q)}
+\mathcal{S}_{S^c,S^c}+\mathcal{S}_{S^f\setminus\iota_{1,\delta^\gamma}(P), S^a}^{(1,0)}+\mathcal{S}_{S^f\setminus\iota_{1,\delta^\gamma}(Q),S^b}^{(1,0)}
\bigg)[\tilde{\varphi}](\tilde{x}),
\end{align*}

\begin{align*}
\mathcal{S}_{D,2}[\tilde{\varphi}](\tilde{x}):&=\bigg(\mathcal{S}_{S^f\setminus\iota_{1,\delta^\gamma}(P), S^a}^{(2,0)}+\mathcal{S}_{S^f\setminus\iota_{1,\delta^\gamma}(Q),S^b}^{(2,0)}
+\mathcal{S}_{\iota_{1,\delta^\gamma}^{(2)}(\tilde{x}),\Gamma_1}^{(2,1)}
+\mathcal{S}_{\iota_{1,\delta^\gamma}^{(1)}(\tilde{x}),\Gamma_2}^{(2,1)}
\\
&\ \ \ \ +\mathcal{S}_{S^a,S^f\setminus\iota_{1,\delta^\gamma}(P)}^{(2,2)}
+\mathcal{S}_{S^b,S^f\setminus\iota_{1,\delta^\gamma}(Q)}^{(2,2)} +\mathcal{S}_{S^a, S^b}^{(2,3)}+\mathcal{S}_{S^b,S^a}^{(2,3)}\bigg)[\tilde{\varphi}](\tilde{x}).
\end{align*}

\end{lem}

\begin{proof}
Owing to the definition formula of single layer potential operator (\ref{eq:layperpt1}) and the Taylor's expansion of $\ln(1+x)$, we have that

\emph{\textbf{Case 1:}} For $y\in\Gamma_{1,\delta}$. Since $\partial D_\delta=S_\delta^a\cup\Gamma_{1,\delta}\cup\Gamma_{2,\delta}\cup S_\delta^b$, it implies
\begin{align*}
&\frac{1}{2\pi}\chi(S_\delta^a)\int_{\Gamma_{1,\delta}\setminus\iota_{\delta^{\gamma}}(P)} \ln|x-y|\varphi(y) d\Gs_{y}\\
=&\frac{1}{4\pi}\chi(S^a)\int_{\Gamma_{1}\setminus\iota_{1,\delta^{\gamma}}(P)} \ln\left(|P-z_{\tilde{y}}|^2\left(1+\frac{2\delta(P-z_{\tilde{y}},\tilde{x}-P)
+2\delta^2(1+(\tilde{x}-P,z_{\tilde{y}}-\tilde{y}))}{|P-z_{\tilde{y}}|^2}\right)\right)\tilde{\varphi}(\tilde{y}) d\tilde{\Gs}_{\tilde{y}}\\
=&\frac{1}{2\pi}\chi(S^a)\int_{\Gamma_{1}\setminus\iota_{1,\delta^{\gamma}}(P)} \ln|P-z_{\tilde{y}}|\tilde{\varphi}(\tilde{y}) d\tilde{\Gs}_{\tilde{y}}\\
&\ \ \ \ \ \ \ \ \ \ \ \ \ +\frac{1}{4\pi}\chi(S^a)\int_{\Gamma_{1}\setminus\iota_{1,\delta^{\gamma}}(P)} \ln\left(1+\frac{2\delta(P-z_{\tilde{y}},\tilde{x}-P)
+2\delta^2(1+(\tilde{x}-P,z_{\tilde{y}}-\tilde{y}))}{|P-z_{\tilde{y}}|^2}\right)\tilde{\varphi}(\tilde{y}) d\tilde{\Gs}_{\tilde{y}}\\
=&\left(\mathcal{S}_{\Gamma_{1}\setminus\iota_{1,\delta^\gamma}(P),S^a}+\delta\mathcal{S}_{\Gamma_{1}\setminus\iota_{1,\delta^\gamma}(P), S^a}^{(1,0)}+\delta^2\mathcal{S}_{\Gamma_{1}\setminus\iota_{1,\delta^\gamma}(P),S^a}^{(2,0)}\right)
[\tilde{\varphi}](\tilde{x})+\mathcal{O}(\delta^3).
\end{align*}
Similarly,
\begin{align*}
&\frac{1}{2\pi}\chi(S_\delta^b)\int_{\Gamma_{1,\delta}\setminus\iota_{\delta^{\gamma}}(Q)} \ln|x-y|\varphi(y) d\Gs_{y}\\
=&\left(\mathcal{S}_{\Gamma_{1}\setminus\iota_{1,\delta^\gamma}(Q),S^b}+\delta\mathcal{S}_{\Gamma_{1}\setminus\iota_{1,\delta^\gamma}(Q), S^b}^{(1,0)}+\delta^2\mathcal{S}_{\Gamma_{1}\setminus\iota_{1,\delta^\gamma}(Q),S^b}^{(2,0)}\right)
[\tilde{\varphi}](\tilde{x})+\mathcal{O}(\delta^3).
\end{align*}
Clearly,
\begin{align*}
\frac{1}{2\pi}\chi(\Gamma_{1,\delta})\int_{\Gamma_{1,\delta}} \ln|x-y|\varphi(y) d\Gs_{y}
=\frac{1}{2\pi}\chi(\Gamma_{1})\int_{\Gamma_{1}}\ln|z_{\tilde{x}}-z_{\tilde{y}}|\tilde{\varphi}(\tilde{y}) d\tilde{\Gs}_{\tilde{y}}.
\end{align*}
Furthermore,
\begin{align*}
&\frac{1}{2\pi}\chi(\Gamma_{2,\delta})\int_{\Gamma_{1,\delta}\setminus\iota_{\delta^{\gamma}}(x)} \ln|x-y|\varphi(y) d\Gs_{y}\\
=&\frac{1}{4\pi}\chi(\Gamma_2)\int_{\Gamma_{1}\setminus\iota_{1,\delta^{\gamma}}(\tilde{x})} \ln\left(|z_{\tilde{x}}-z_{\tilde{y}}|^2\left(1+\frac{4\delta^2}{|z_{\tilde{x}}-z_{\tilde{y}}|^2}\right)\right)\tilde{\varphi}(\tilde{y}) d\tilde{\Gs}_{\tilde{y}},\\
=&\left(\mathcal{S}_{\Gamma_{1}\setminus\iota_{1,\delta^\gamma}^{(1)}(\tilde{x}),\Gamma_2}
+\delta^2\mathcal{S}_{\Gamma_{1}\setminus\iota_{1,\delta^\gamma}(\tilde{x}),\Gamma_2}^{(2,1)}\right)
[\tilde{\varphi}](\tilde{x})+\mathcal{O}(\delta^4).
\end{align*}

\emph{\textbf{Case 2:}} For $y\in\Gamma_{2,\delta}$. Noticing the symmetry of geometry of $\partial D$, we only need to exchange $\Gamma_{2,\delta}$ with $\Gamma_{1,\delta}$ in \emph{Case 1} and get the corresponding expansion formula. That is,
\begin{align*}
&\frac{1}{2\pi}\chi(S_\delta^a)\int_{\Gamma_{2,\delta}\setminus\iota_{\delta^{\gamma}}(P)} \ln|x-y|\varphi(y) d\Gs_{y}\\
=&\left(\mathcal{S}_{\Gamma_{2}\setminus\iota_{1,\delta^\gamma}(P),S^a}+\delta\mathcal{S}_{\Gamma_{2}\setminus\iota_{1,\delta^\gamma}(P), S^a}^{(1,0)}+\delta^2\mathcal{S}_{\Gamma_{2}\setminus\iota_{1,\delta^\gamma}(P),S^a}^{(2,0)}\right)
[\tilde{\varphi}](\tilde{x})+\mathcal{O}(\delta^3),
\end{align*}
\begin{align*}
&\frac{1}{2\pi}\chi(S_\delta^b)\int_{\Gamma_{2,\delta}\setminus\iota_{\delta^{\gamma}}(Q)} \ln|x-y|\varphi(y) d\Gs_{y}\\
=&\left(\mathcal{S}_{\Gamma_{2}\setminus\iota_{1,\delta^\gamma}(Q),S^b}+\delta\mathcal{S}_{\Gamma_{2}\setminus\iota_{1,\delta^\gamma}(Q), S^b}^{(1,0)}+\delta^2\mathcal{S}_{\Gamma_{2}\setminus\iota_{1,\delta^\gamma}(Q),S^b}^{(2,0)}\right)
[\tilde{\varphi}](\tilde{x})+\mathcal{O}(\delta^3),
\end{align*}
\begin{align*}
\frac{1}{2\pi}\chi(\Gamma_{2,\delta})\int_{\Gamma_{2,\delta}} \ln|x-y|\varphi(y) d\Gs_{y}
=\frac{1}{2\pi}\chi(\Gamma_{2})\int_{\Gamma_{2}}\ln|z_{\tilde{x}}-z_{\tilde{y}}|\tilde{\varphi}(\tilde{y}) d\tilde{\Gs}_{\tilde{y}},
\end{align*}
and
\begin{align*}
\frac{1}{2\pi}\chi(\Gamma_{1,\delta})\int_{\Gamma_{2,\delta}\setminus\iota_{\delta^{\gamma}}(x)} \ln|x-y|\varphi(y) d\Gs_{y}
=\left(\mathcal{S}_{\Gamma_{2}\setminus\iota_{1,\delta^\gamma}^{(1)}(\tilde{x}),\Gamma_1}
+\delta^2\mathcal{S}_{\Gamma_{2}\setminus\iota_{1,\delta^\gamma}(\tilde{x}),\Gamma_1}^{(2,1)}\right)
[\tilde{\varphi}](\tilde{x})+\mathcal{O}(\delta^3).
\end{align*}

\emph{\textbf{Case 3:}} For $y\in S_\delta^a$.
\begin{align*}
&\frac{1}{2\pi}\chi(\Gamma_{1,\delta}\setminus\iota_{\delta^{\gamma}}(P))\int_{S_\delta^a} \ln|x-y|\varphi(y) d\Gs_{y}\\
=&\frac{1}{4\pi}\chi(\Gamma_{1}\setminus\iota_{1,\delta^{\gamma}}(P))\int_{S^a} \ln\left(|P-z_{\tilde{x}}|^2\left(1+\frac{2\delta(P-z_{\tilde{x}},\tilde{y}-P)
+2\delta^2(1+(\tilde{y}-P,z_{\tilde{x}}-\tilde{x}))}{|P-z_{\tilde{x}}|^2}\right)\right)\tilde{\varphi}(\tilde{y}) \delta d\tilde{\Gs}_{\tilde{y}}\\
=&\frac{1}{2\pi}\chi(\Gamma_{1}\setminus\iota_{1,\delta^{\gamma}}(P))\delta\int_{S^a} \ln|P-z_{\tilde{x}}|\tilde{\varphi}(\tilde{y}) d\tilde{\Gs}_{\tilde{y}}\\
&\ \ \ \ \ \ \ \ \ \ \ \ \ +\frac{1}{4\pi}\chi(\Gamma_{1}\setminus\iota_{1,\delta^{\gamma}}(P))\delta\int_{S^a} \ln\left(1+\frac{2\delta(P-z_{\tilde{x}},\tilde{y}-P)
+2\delta^2(1+(\tilde{x}-P,z_{\tilde{x}}-\tilde{x}))}{|P-z_{\tilde{x}}|^2}\right)\tilde{\varphi}(\tilde{y}) d\tilde{\Gs}_{\tilde{y}}\\
=&\left(\delta\mathcal{S}_{S^a,\Gamma_{1}\setminus\iota_{1,\delta^\gamma}(P)}+\delta^2\mathcal{S}_{S^a,\Gamma_{1}\setminus\iota_{1,\delta^\gamma}(P)}^{(2,2)}\right)
[\tilde{\varphi}](\tilde{x})+\mathcal{O}(\delta^3).
\end{align*}
By exchanging $\Gamma_{2,\delta}$ with $\Gamma_{1,\delta}$, we have
\begin{align*}
\frac{1}{2\pi}\chi(\Gamma_{2,\delta}\setminus\iota_{\delta^{\gamma}}(P))\int_{S_\delta^a} \ln|x-y|\varphi(y) d\Gs_{y}
=\left(\delta\mathcal{S}_{S^a,\Gamma_{2}\setminus\iota_{1,\delta^\gamma}(P)}
+\delta^2\mathcal{S}_{S^a,\Gamma_{2}\setminus\iota_{1,\delta^\gamma}(P)}^{(2,2)}\right)
[\tilde{\varphi}](\tilde{x})+\mathcal{O}(\delta^3).
\end{align*}
Moreover, we also obtain
\begin{align*}
\frac{1}{2\pi}\chi(S_\delta^b)\int_{S_\delta^a} \ln|x-y|\varphi(y) d\Gs_{y}
=\left(\delta\mathcal{S}_{S^a,S^b}+\delta^2\mathcal{S}_{S^a,S^b}^{(2,3)}\right)
[\tilde{\varphi}](\tilde{x})+\mathcal{O}(\delta^3),
\end{align*}
and
\begin{align*}
\frac{1}{2\pi}\chi(S_\delta^a)\int_{S_\delta^a} \ln|x-y|\varphi(y) d\Gs_{y}
=\left(\delta\ln\delta\Pi_{S^a,S^a}^{(0)}+\delta\mathcal{S}_{S^a,S^a}\right)
[\tilde{\varphi}](\tilde{x}).
\end{align*}

\emph{\textbf{Case 4:}} For $y\in S_\delta^b$. Similar to \emph{Case 3}, it follows that
\begin{align*}
\frac{1}{2\pi}\chi(\Gamma_{1,\delta}\setminus\iota_{\delta^{\gamma}}(Q))\int_{S_\delta^b} \ln|x-y|\varphi(y) d\Gs_{y}
=\left(\delta\mathcal{S}_{S^b,\Gamma_{1}\setminus\iota_{1,\delta^\gamma}(Q)}
+\delta^2\mathcal{S}_{S^b,\Gamma_{1}\setminus\iota_{1,\delta^\gamma}(Q)}^{(2,2)}\right)
[\tilde{\varphi}](\tilde{x})+\mathcal{O}(\delta^3),
\end{align*}
\begin{align*}
\frac{1}{2\pi}\chi(\Gamma_{2,\delta}\setminus\iota_{\delta^{\gamma}}(Q))\int_{S_\delta^b} \ln|x-y|\varphi(y) d\Gs_{y}
=\left(\delta\mathcal{S}_{S^b,\Gamma_{2}\setminus\iota_{1,\delta^\gamma}(Q)}
+\delta^2\mathcal{S}_{S^b,\Gamma_{2}\setminus\iota_{1,\delta^\gamma}(Q)}^{(2,2)}\right)
[\tilde{\varphi}](\tilde{x})+\mathcal{O}(\delta^3),
\end{align*}
\begin{align*}
\frac{1}{2\pi}\chi(S_\delta^a)\int_{S_\delta^b} \ln|x-y|\varphi(y) d\Gs_{y}
=\left(\delta\mathcal{S}_{S^b,S^a}+\delta^2\mathcal{S}_{S^b,S^a}^{(2,3)}\right)
[\tilde{\varphi}](\tilde{x})+\mathcal{O}(\delta^3),
\end{align*}
and
\begin{align*}
\frac{1}{2\pi}\chi(S_\delta^b)\int_{S_\delta^b} \ln|x-y|\varphi(y) d\Gs_{y}
=\left(\delta\ln\delta\Pi_{S^b,S^b}^{(0)}+\delta\mathcal{S}_{S^b,S^b}\right)
[\tilde{\varphi}](\tilde{x}).
\end{align*}

Next, we deal with the following terms:
\begin{align*}
\left\|\mathcal{S}_{\iota_{\delta^\gamma}^{(1)}(x), S_\delta^a}^{(0,1)}\right\|
_{\mathcal{L}\left(H^{-\frac{1}{2}}(\iota_{\delta^\gamma}^{(1)}(P)),H^{\frac{1}{2}}(S_\delta^a)\right)}
=&\sup_{\varphi\neq0}\frac{\frac{1}{2\pi}
\left\|\int_{\iota_{\delta^\gamma}^{(1)}(x)}\ln|x-y|\varphi(y)d\sigma_y\right\|
_{H^{\frac{1}{2}}(S_\delta^a)}}
{\|\varphi\|_{H^{-\frac{1}{2}}(\iota_{\delta^\gamma}^{(1)}(P))}}\\
\geq&\sup_{\varphi\neq0}\frac{\frac{1}{2\pi}
\left\|\int_{\iota_{\delta^\gamma}^{(1)}(x)}\ln|x-y|\varphi(y)d\sigma_y\right\|
_{L^{2}(S_\delta^a)}}
{\|\varphi\|_{L^{2}(\iota_{\delta^\gamma}^{(1)}(P))}}\\
\geq&\frac{\frac{1}{2\pi}
\left\|\int_{\iota_{\delta^\gamma}^{(1)}(x)}\ln|x-y|d\sigma_y\right\|
_{L^{2}(S_\delta^a)}}
{\|1\|_{L^{2}(\iota_{\delta^\gamma}^{(1)}(P))}}\\
=&\frac{1}{2\pi}\frac{1}{\sqrt{\delta^\gamma+\pi\delta}}
\int_{S_\delta^a}\left|\int_{\iota_{\delta^\gamma}^{(1)}(P)}
\ln|x-y|d\sigma_y\right|d\sigma_x.
\end{align*}
Then, by using the polar coordinate transformation and mean-value theorem, it deduces
\begin{align*}
&\frac{1}{2\pi}\frac{1}{\sqrt{\delta^\gamma+\pi\delta}}
\int_{S_\delta^a}\left|\int_{\iota_{\delta^\gamma}^{(1)}(P)}
\ln|x-y|d\sigma_y\right|d\sigma_x\\
=&\frac{1}{2\pi}\frac{1}{\sqrt{\delta^\gamma+\pi\delta}}
\frac{1}{2\pi}\int_{\frac{\pi}{2}}^{\frac{3\pi}{2}}\left|\int_{-\frac{L}{2}}^{-\frac{L}{2}+\delta^\gamma}
\ln\sqrt{\left(-\frac{L}{2}+\delta\cos\theta-y_1\right)^2+\delta^2(1+\sin\theta)^2}d\sigma_y\right|d\theta\\
=&\frac{1}{8\pi}\frac{1}{\sqrt{\delta^\gamma+\pi\delta}}
\int_{-\frac{L}{2}}^{-\frac{L}{2}+\delta^\gamma}
\ln\left(\left(-\frac{L}{2}+\delta\cos\tilde{\theta}-y_1\right)^2+\delta^2(1+\sin\tilde{\theta})^2\right)d\sigma_y,
\ \ \tilde{\theta}\in\left(\frac{\pi}{2},\frac{3\pi}{2}\right).
\end{align*}
Thus, by using the integral formula of $\int\ln(a^2+x^2)dx$, we obtain
\begin{align*}
\left\|\mathcal{S}_{\iota_{\delta^\gamma}^{(1)}(x), S_\delta^a}^{(0,1)}\right\|
_{\mathcal{L}\left(H^{-\frac{1}{2}}(\iota_{\delta^\gamma}^{(1)}(P)),H^{\frac{1}{2}}(S_\delta^a)\right)}
\geq\mathcal{O}\left(\delta^{\frac{\gamma}{2}}\ln\delta\right).
\end{align*}
Similarly, we can estimate $\left\|\mathcal{S}_{\iota_{\delta^\gamma}^{(2)}(x), S_\delta^a}^{(0,1)}\right\|
_{\mathcal{L}\left(H^{-\frac{1}{2}}(\iota_{\delta^\gamma}^{(2)}(P)),H^{\frac{1}{2}}(S_\delta^a)\right)}$,
$\left\|\mathcal{S}_{\iota_{\delta^\gamma}^{(1)}(x), S_\delta^b}^{(0,1)}\right\|
_{\mathcal{L}\left(H^{-\frac{1}{2}}(\iota_{\delta^\gamma}^{(1)}(P)),H^{\frac{1}{2}}(S_\delta^b)\right)}$, and
$\left\|\mathcal{S}_{\iota_{\delta^\gamma}^{(2)}(x), S_\delta^b}^{(0,1)}\right\|
_{\mathcal{L}\left(H^{-\frac{1}{2}}(\iota_{\delta^\gamma}^{(2)}(P)),H^{\frac{1}{2}}(S_\delta^b)\right)}$ is also no less than $\mathcal{O}\left(\delta^{\frac{\gamma}{2}}\ln\delta\right)$.

Moreover,
\begin{align*}
\left\|\mathcal{S}_{S_\delta^a,\iota_{\delta^\gamma}^{(1)}(x)}^{(0,2)}\right\|
_{\mathcal{L}\left(H^{-\frac{1}{2}}(S_\delta^a),H^{\frac{1}{2}}(\iota_{\delta^\gamma}^{(1)}(P))\right)}
=&\sup_{\varphi\neq0}\frac{\frac{1}{2\pi}
\left\|\int_{S_\delta^a}\ln|x-y|\varphi(y)d\sigma_y\right\|
_{H^{\frac{1}{2}}(\iota_{\delta^\gamma}^{(1)}(P))}}
{\|\varphi\|_{H^{-\frac{1}{2}}(S_\delta^a)}}\\
\geq&\sup_{\varphi\neq0}\frac{\frac{1}{2\pi}
\left\|\int_{S_\delta^a}\ln|x-y|d\sigma_y\right\|
_{L^{2}(\iota_{\delta^\gamma}^{(1)}(P))}}
{\|1\|_{L^{2}(S_\delta^a)}}\\
=&\frac{1}{2\pi}\frac{1}{\sqrt{\pi\delta}}
\int_{\iota_{\delta^\gamma}^{(1)}(x)}\left|\int_{S_\delta^a}
\ln|x-y|d\sigma_y\right|d\sigma_x\\
\geq&\mathcal{O}\left(\delta^{\gamma-\frac{1}{2}}\ln\delta\right).
\end{align*}
Likewise, the terms: $\left\|\mathcal{S}_{S_\delta^a,\iota_{\delta^\gamma}^{(2)}(x)}^{(0,2)}\right\|
_{\mathcal{L}\left(H^{-\frac{1}{2}}(S_\delta^a),H^{\frac{1}{2}}(\iota_{\delta^\gamma}^{(2)}(P))\right)}$,
$\left\|\mathcal{S}_{S_\delta^b,\iota_{\delta^\gamma}^{(1)}(x)}^{(0,2)}\right\|
_{\mathcal{L}\left(H^{-\frac{1}{2}}(S_\delta^b),H^{\frac{1}{2}}(\iota_{\delta^\gamma}^{(1)}(P))\right)}$, and\\
$\left\|\mathcal{S}_{S_\delta^b,\iota_{\delta^\gamma}^{(2)}(x)}^{(0,2)}\right\|
_{\mathcal{L}\left(H^{-\frac{1}{2}}(S_\delta^b),H^{\frac{1}{2}}(\iota_{\delta^\gamma}^{(2)}(P))\right)}$
can be also estimated and no less than $\mathcal{O}\left(\delta^{\gamma-\frac{1}{2}}\ln\delta\right)$.

Finally,
\begin{align*}
\left\|\mathcal{S}_{\iota_{\delta^\gamma}^{(1)}(x),\Gamma_{2,\delta}}^{(0,1)}\right\|
_{\mathcal{L}\left(H^{-\frac{1}{2}}(\iota_{\delta^\gamma}^{(1)}(x)),H^{\frac{1}{2}}(\Gamma_{2,\delta})\right)}
=&\sup_{\varphi\neq0}\frac{\frac{1}{2\pi}
\left\|\int_{\iota_{\delta^\gamma}^{(1)}(x)}\ln|x-y|\varphi(y)d\sigma_y\right\|
_{H^{\frac{1}{2}}(\Gamma_{2,\delta})}}
{\|\varphi\|_{H^{-\frac{1}{2}}(\iota_{\delta^\gamma}^{(1)}(x))}}\\
\geq&\frac{\frac{1}{2\pi}
\left\|\int_{\iota_{\delta^\gamma}^{(1)}(x)}\ln|x-y|d\sigma_y\right\|
_{L^{2}(\Gamma_{2,\delta})}}
{\|1\|_{L^2(\iota_{\delta^\gamma}^{(1)}(x))}}\\
\geq&\mathcal{O}\left(\delta^{\frac{\gamma}{2}}\ln\delta\right).
\end{align*}
By using the same method, we can also find that
\begin{align*}
\left\|\mathcal{S}_{\iota_{\delta^\gamma}^{(2)}(\tilde{x}),\Gamma_{1,\delta}}^{(0,1)}\right\|
\geq\mathcal{O}\left(\delta^{\frac{\gamma}{2}}\ln\delta\right).
\end{align*}
Hence, (\ref{singexp1}) holds.
\end{proof}

Let the boundary integral operators $\mathcal{K}_{D_\delta,1}$, $\mathcal{S}_{D_\delta,1}$ and $\mathcal{S}_{D_\delta,2}$ be respectively defined by
\begin{equation}\label{laypot01}
\begin{split}
\mathcal{S}_{D_\delta,1}[\varphi](\Bx):&=-\frac{1}{8\pi}\int_{\p D_\delta} |\Bx-\By|^2 \varphi(\By) d\sigma(\By),\\
\mathcal{S}_{D_\delta,2}[\varphi](\Bx):&=-\frac{1}{8\pi}\int_{\p D_\delta} (\tau_k+\ln|\Bx-\By|)|\Bx-\By|^2 \varphi(\By) d\sigma(\By),\\
\mathcal{K}_{D_\delta,1}[\varphi](\Bx):&=-\frac{1}{8\pi}\int_{\p D_\delta} \frac{\partial|\Bx-\By|^2}{\partial\nu(\Bx)} \varphi(\By) d\sigma(\By).
\end{split}
\end{equation}
Furthermore, we introduce the following boundary integral operators:
\beq\label{laypot03}
\begin{split}
\mathcal{S}_{\Lambda_1,\Lambda_2,0}[\tilde{\varphi}](\tilde{\Bx}):&=-\frac{1}{8\pi}\chi(\Lambda_2)\int_{\Lambda_1} |\Bz_{\tilde{\Bx}}-\Bz_{\tilde{\By}}|^2\tilde{\varphi}(\tilde{\By}) d\tilde{\sigma}(\tilde{\By}),\\
\mathcal{S}_{\Lambda_1,\Lambda_2,1}[\tilde{\varphi}](\tilde{\Bx}):&=-\frac{1}{4\pi}\chi(\Lambda_2)\int_{\Lambda_1} (\Bz_{\tilde{\Bx}}-\Bz_{\tilde{\By}},\tilde{\Bx}-\Bz_{\tilde{\Bx}})\tilde{\varphi}(\tilde{\By}) d\tilde{\sigma}(\tilde{\By}),\\
\mathcal{K}_{\Lambda_1,\Lambda_2,0}[\tilde{\varphi}](\tilde{\Bx}):&=-\frac{1}{4\pi}\chi(\Lambda_2)\int_{\Lambda_1} (\Bz_{\tilde{\Bx}}-\Bz_{\tilde{\By}},\nu_{\tilde{\Bx}})\tilde{\varphi}(\tilde{\By}) d\tilde{\sigma}(\tilde{\By}),\\
\mathcal{K}_{\Lambda_1,\Lambda_2,1}[\tilde{\varphi}](\tilde{\Bx}):&=-\frac{1}{4\pi}\chi(\Lambda_2)\int_{\Lambda_1} \left(1+(\nu_{\tilde{\Bx}},\Bz_{\tilde{\By}}-\tilde{\By})\right)\tilde{\varphi}(\tilde{\By}) d\tilde{\sigma}(\tilde{\By}).
\end{split}
\eeq

\begin{lem}\label{le:app0101}
Let $\varphi\in\mathcal {H}^*(\partial D_\delta)$ and $\tilde{\varphi}(\tilde{\Bx})=\varphi(\Bx)$ for $x\in\partial D_\delta$ and $\tilde{\Bx}\in\partial D$. Then the following asymptotic results hold:
\begin{enumerate}
\item[(i)] \begin{equation}\label{sd1}
\mathcal{S}_{D_\delta,1}[\varphi](\Bx)=\mathcal{S}_{D,1,0}[\tilde{\varphi}](\tilde{\Bx})
+\delta\mathcal{S}_{D,1,1}[\tilde{\varphi}](\tilde{\Bx})
+\mathcal{O}(\delta^2),
\end{equation}

\item[(ii)]
\begin{equation}\label{kd1}
\begin{split}
&\mathcal{K}_{D_\delta,1}[\varphi](\Bx)=\mathcal{K}_{D,1,0}[\tilde{\varphi}](\tilde{\Bx})\\
+&\delta\left(\mathcal{K}_{D,1,1}[\tilde{\varphi}](\tilde{\Bx})+\chi(\Gamma_2)\frac{1}{2\pi}\int_{\Gamma_1} \tilde\varphi(\tilde{\By})d\tilde{\sigma}(\tilde{\By})+\chi(\Gamma_1)\frac{1}{2\pi}\int_{\Gamma_2} \tilde\varphi(\tilde{\By})d\tilde{\sigma}(\tilde{\By})\right)
+\mathcal{O}(\delta^2),
\end{split}
\end{equation}
where
\begin{align*}
&\mathcal{S}_{D,1,0}[\tilde{\varphi}](\tilde{\Bx})=\mathcal{S}_{S^f,\partial D,0}[\tilde{\varphi}](\tilde{\Bx}),\\
&\mathcal{S}_{D,1,1}[\tilde{\varphi}](\tilde{\Bx})=\bigg(\mathcal{S}_{S^a,\partial D\setminus S^a,0}+\mathcal{S}_{S^b,\partial D\setminus S^b,0}+\mathcal{S}_{S^f,S^c,1}\bigg)[\tilde{\varphi}](\tilde{\Bx}),\\
&\mathcal{K}_{D,1,0}[\tilde{\varphi}](\tilde{\Bx})=\mathcal{K}_{S^f,S^c,0}[\tilde{\varphi}](\tilde{\Bx}),\\
&\mathcal{K}_{D,1,1}[\tilde{\varphi}](\tilde{\Bx})=\bigg(\mathcal{K}_{S^a,S^b,0}+\mathcal{K}_{S^b,S^a,0}
+\mathcal{K}_{S^f,S^c,1}+\mathcal{S}_{S^c,S^f,1}\bigg)[\tilde{\varphi}](\tilde{\Bx}).
\end{align*}
\end{enumerate}
\end{lem}
\begin{proof}
\item[(i)]
From the definition of $\mathcal{S}_{D_\delta,1}$, we have

\emph{\textbf{Case 1:}} For $\By\in\Gamma_{1,\delta}$, it follows that
\begin{align*}
|\Bx-\By|^2&=\left(P-\Bz_{\tilde{\By}}+\delta(\tilde{\Bx}-P)+\delta(\Bz_{\tilde{\By}}-\tilde{\By}),
P-\Bz_{\tilde{\By}}+\delta(\tilde{\Bx}-P)+\delta(\Bz_{\tilde{\By}}-\tilde{\By})\right)\\
&=|P-\Bz_{\tilde{\By}}|^2+2\delta(P-\Bz_{\tilde{y}},\tilde{\Bx}-P)+2\delta^2
(1+(\tilde{\Bx}-P,\Bz_{\tilde{\By}}-\tilde{\By})).
\end{align*}
Since $\partial D_\delta=S_\delta^a\cup\Gamma_{1,\delta}\cup\Gamma_{2,\delta}\cup S_\delta^b$, it implies
\begin{align*}
&-\frac{1}{8\pi}\chi(S_\delta^a)\int_{\Gamma_{1,\delta}} |\Bx-\By|^2\varphi(y) d\sigma(\By)\\
=&-\frac{1}{8\pi}\chi(S^a)\int_{\Gamma_{1}}|P-\Bz_{\tilde{y}}|^2\tilde{\varphi}(\tilde{\By}) d\tilde{\sigma}(\tilde{\By})-\frac{1}{4\pi}\delta\chi(S^a)\int_{\Gamma_{1}} (P-\Bz_{\tilde{\By}},\tilde{\Bx}-P)\tilde{\varphi}(\tilde{\By}) d\tilde{\sigma}(\tilde{\By})+\mathcal{O}(\delta^2).
\end{align*}
Similarly,
\begin{align*}
&-\frac{1}{8\pi}\chi(S_\delta^b)\int_{\Gamma_{1,\delta}} |\Bx-\By|^2\varphi(\By) d\sigma(\By)\\
=&-\frac{1}{8\pi}\chi(S^b)\int_{\Gamma_{1}}|Q-\Bz_{\tilde{\By}}|^2\tilde{\varphi}(\tilde{\By}) d\tilde{\sigma}(\tilde{\By})-\frac{1}{4\pi}\delta\chi(S^b)\int_{\Gamma_{1}} (Q-\Bz_{\tilde{\By}},\tilde{\Bx}-Q)\tilde{\varphi}(\tilde{\By}) d\tilde{\sigma}(\tilde{\By})+\mathcal{O}(\delta^2).
\end{align*}
Clearly,
\begin{align*}
-\frac{1}{8\pi}\chi(\Gamma_{1,\delta})\int_{\Gamma_{1,\delta}}|\Bx-\By|^2\varphi(\By) d\sigma(\By)
=-\frac{1}{8\pi}\chi(\Gamma_{1})\int_{\Gamma_{1}}|\Bz_{\tilde{\Bx}}-\Bz_{\tilde{\By}}|^2\tilde{\varphi}(\tilde{\By}) d\tilde{\sigma}(\tilde{\By}).
\end{align*}
Furthermore,
\begin{align*}
-\frac{1}{8\pi}\chi(\Gamma_{2,\delta})\int_{\Gamma_{1,\delta}\setminus\iota_{\delta^{\gamma}}(\Bx)} |\Bx-\By|^2\varphi(y) d\sigma(\By)
=-\frac{1}{8\pi}\chi(\Gamma_2)\int_{\Gamma_{1}} |\Bz_{\tilde{\Bx}}-\Bz_{\tilde{\By}}|^2\tilde{\varphi}(\tilde{\By}) d\tilde{\sigma}(\tilde{\By})+\mathcal{O}(\delta^2).
\end{align*}

\emph{\textbf{Case 2:}} For $\By\in\Gamma_{2,\delta}$. Notice the symmetry of geometry of $\partial D_\delta$, we only need to exchange $\Gamma_{2,\delta}$ with $\Gamma_{1,\delta}$ in \emph{Case 1} and get the corresponding expand formula. Namely,
\begin{align*}
&-\frac{1}{8\pi}\chi(S_\delta^a)\int_{\Gamma_{2,\delta}} |\Bx-\By|^2\varphi(y) d\sigma(\By)\\
=&-\frac{1}{8\pi}\chi(S^a)\int_{\Gamma_{2}}|P-\Bz_{\tilde{\By}}|^2\tilde{\varphi}(\tilde{\By}) d\tilde{\sigma}(\tilde{\By})-\frac{1}{4\pi}\delta\chi(S^a)\int_{\Gamma_{2}} (P-\Bz_{\tilde{\By}},\tilde{\Bx}-P)\tilde{\varphi}(\tilde{\By}) d\tilde{\sigma}(\tilde{\By})+\mathcal{O}(\delta^2),
\end{align*}
\begin{align*}
&-\frac{1}{8\pi}\chi(S_\delta^b)\int_{\Gamma_{2,\delta}} |\Bx-\By|^2\varphi(\By) d\sigma(\By)\\
=&-\frac{1}{8\pi}\chi(S^b)\int_{\Gamma_{2}}|Q-\Bz_{\tilde{y}}|^2\tilde{\varphi}(\tilde{\By}) d\tilde{\sigma}(\tilde{\By})-\frac{1}{4\pi}\delta\chi(S^b)\int_{\Gamma_{2}} (Q-\Bz_{\tilde{\By}},\tilde{\Bx}-Q)\tilde{\varphi}(\tilde{\By}) d\tilde{\sigma}(\tilde{\By})+\mathcal{O}(\delta^2),
\end{align*}
\begin{align*}
&-\frac{1}{8\pi}\chi(\Gamma_{2,\delta})\int_{\Gamma_{2,\delta}}|\Bx-\By|^2\varphi(\By) d\sigma(\By)
=-\frac{1}{8\pi}\chi(\Gamma_{2})\int_{\Gamma_{2}}|\Bz_{\tilde{\Bx}}-\Bz_{\tilde{\By}}|^2\tilde{\varphi}(\tilde{\By}) d\tilde{\sigma}(\tilde{\By}),
\end{align*}
and
\begin{align*}
-\frac{1}{8\pi}\chi(\Gamma_{1,\delta})\int_{\Gamma_{2,\delta}} |\Bx-\By|^2\varphi(\By) d\sigma(\By)
=-\frac{1}{8\pi}\chi(\Gamma_1)\int_{\Gamma_{2}} |\Bz_{\tilde{\Bx}}-\Bz_{\tilde{\By}}|^2\tilde{\varphi}(\tilde{\By}) d\tilde{\sigma}(\tilde{\By})+\mathcal{O}(\delta^2).
\end{align*}

\emph{\textbf{Case 3:}} For $\By\in S_\delta^a$.
\begin{align*}
-\frac{1}{8\pi}\chi(\Gamma_{1,\delta})\int_{S_\delta^a} |\Bx-\By|^2\varphi(\By) d\sigma(\By)
=-\frac{1}{8\pi}\delta\chi(\Gamma_{1})\int_{S^a}|P-\Bz_{\tilde{\By}}|^2\tilde{\varphi}(\tilde{\By}) d\tilde{\sigma}(\tilde{\By})+\mathcal{O}(\delta^2).
\end{align*}
Similarly,
\begin{align*}
-\frac{1}{8\pi}\chi(\Gamma_{2,\delta})\int_{S_\delta^a} |\Bx-\By|^2\varphi(\By) d\sigma(\By)
=-\frac{1}{8\pi}\delta\chi(\Gamma_{2})\int_{S^a}|P-\Bz_{\tilde{\By}}|^2\tilde{\varphi}(\tilde{\By}) d\tilde{\sigma}(\tilde{\By})+\mathcal{O}(\delta^2).
\end{align*}
Clearly,
\begin{align*}
-\frac{1}{8\pi}\chi(S_\delta^b)\int_{S_\delta^a}|\Bx-\By|^2\varphi(\By) d\sigma(\By)
=-\frac{1}{8\pi}\delta\chi(S^b)\int_{S^a}|P-Q|^2\tilde{\varphi}(\tilde{\By}) d\tilde{\sigma}(\tilde{\By})+\mathcal{O}(\delta^2).
\end{align*}
Furthermore,
\begin{align*}
-\frac{1}{8\pi}\chi(S_\delta^a)\int_{S_\delta^a}|\Bx-\By|^2\varphi(\By) d\sigma(\By)
=\mathcal{O}(\delta^3).
\end{align*}

\emph{\textbf{Case 4:}} For $\By\in S_\delta^b$. Owing to the symmetry of geometry of $\partial D_\delta$, and then exchange $S_\delta^a$ with $S_\delta^b$ in \emph{Case 3}, we get
\begin{align*}
-\frac{1}{8\pi}\chi(\Gamma_{1,\delta})\int_{S_\delta^b} |\Bx-\By|^2\varphi(\By) d\sigma(\By)
=-\frac{1}{8\pi}\delta\chi(\Gamma_{1})\int_{S^b}|P-\Bz_{\tilde{\By}}|^2\tilde{\varphi}(\tilde{\By}) d\tilde{\sigma}(\tilde{\By})+\mathcal{O}(\delta^2),
\end{align*}
\begin{align*}
-\frac{1}{8\pi}\chi(\Gamma_{2,\delta})\int_{S_\delta^b} |\Bx-\By|^2\varphi(\By) d\sigma(\By)
=-\frac{1}{8\pi}\delta\chi(\Gamma_{2})\int_{S^b}|P-\Bz_{\tilde{\By}}|^2\tilde{\varphi}(\tilde{\By}) d\tilde{\sigma}(\tilde{\By})+\mathcal{O}(\delta^2),
\end{align*}
\begin{align*}
-\frac{1}{8\pi}\chi(S_\delta^a)\int_{S_\delta^b}|\Bx-\By|^2\varphi(\By) d\sigma(\By)
=-\frac{1}{8\pi}\delta\chi(S^b)\int_{S^a}|P-Q|^2\tilde{\varphi}(\tilde{\By}) d\tilde{\sigma}(\tilde{\By})+\mathcal{O}(\delta^2),
\end{align*}
and
\begin{align*}
-\frac{1}{8\pi}\chi(S_\delta^b)\int_{S_\delta^b}|\Bx-\By|^2\varphi(\By) d\sigma(\By)
=\mathcal{O}(\delta^3).
\end{align*}
Hence, by combining above four case, (\ref{sd1}) is proved.

\item[(ii)] By using the formula of $(\Bx-\By,\nu_x)$ in $\mathcal{K}_{D_\delta,1}$ and the same method in (i), (\ref{kd1}) can be proved.
\end{proof}

Finally, we give the asymptotic expansions of layer potential operators with respect to the frequency. Suppose $k\ll 1$, one can first find out that \cite{AMRZ14}
\beq\label{eq:asyHank01}
H_0^{(1)}(k|\Bx|)=c_k + \mathrm{i}\frac{2}{\pi}\ln|\Bx|-\frac{\mathrm{i}}{2\pi}(k^2\ln k)|\Bx|^2-\frac{\mathrm{i}}{2\pi}k^2|\Bx|^2(\ln|\Bx|+\tau_k)+\Ocal(k^4\ln k),
\eeq
where $c_k:=1+\mathrm{i}\frac{2}{\pi}\Big(\ln\frac{k}{2}+\gamma\Big)$, $\tau_k:=1+\ln2+\frac{\mathrm{i}\pi}{2}-\gamma$, and $\gamma=0.5772...$ is the Euler-Mascheroni constant.
Based on \eqnref{eq:asyHank01}, one can then easily find the asymptotic expansions for single layer potential operator $\Scal_B^k$ and Neumann-Poincar\'e operator $(\Kcal_B^k)^*$ for $k\ll 1$ as follows:
\begin{align}\label{eq:asymplayer01}
\Scal_B^k[\varphi]=&\Scal_B^0[\varphi] -\frac{\mathrm{i}}{4} c_k \la 1, \varphi\ra_{\p B}+(k^2\ln k) \mathcal{S}_{B,1}[\varphi]+k^2 \mathcal{S}_{B,2}[\varphi]+ k^4\ln k\Ocal(\|\varphi\|),\\
(\Kcal_B^k)^*[\varphi]=&(\Kcal_B^0)^*[\varphi]+(k^2\ln k) \mathcal{K}_{B,1}[\varphi]+k^2\Ocal(\|\varphi\|),   \quad \varphi\in\mathcal{H}^*(\p B),\label{eq:asymplayer02}
\end{align}
where $\mathcal{S}_{B,1}$, $\mathcal{S}_{B,2}$ and $\mathcal{K}_{B,1}$ are defined as (\ref{laypot01}).
\section{Quantitative analysis of the scattered field}

In this section, we focus on the quantitative analysis of the scattered field $u^s$. Define $k_c:=\omega\sqrt{\varepsilon_c}$. With the help of layer potential techniques, one has the following integral representation for the solution to \eqnref{eq:helm01}:
\begin{align}\label{solu01}
u(\Bx)=
\begin{cases}
u^i(x)+\mathcal {S}_{D_\delta}^{\omega}[\psi](\Bx), & \Bx\in
\mathbb{R}^2\setminus\overline{{D_\delta}},\\
\mathcal {S}_{D_\delta}^{k_c}[\varphi](\Bx), &\Bx\in D_\delta,
\end{cases}
\end{align}
where by using the jump formula \eqnref{eq:trace}, $(\varphi,\psi)\in H^{-\frac{1}{2}}(\partial {D_\delta})\times H^{-\frac{1}{2}}(\partial {D_\delta})$ satisfy the following integral system:
\begin{equation}\label{integ01}
\begin{cases}
\mathcal {S}_{D_\delta}^{k_c}[\varphi]-\mathcal {S}_{D_\delta}^{\omega}[\psi]=u^i &\text{on}\ \partial {D_\delta},\medskip\\
\frac{1}{\varepsilon_c}\left(-\frac{1}{2}\mathcal {I}+(\mathcal
{K}_{D_\delta}^{k_c})^*\right)[\varphi]
-\left(\frac{1}{2}\mathcal {I}+(\mathcal
{K}_{D_\delta}^{\omega})^*\right)[\psi] =\frac{\partial u^i}{\partial\nu} &\text{on}\ \partial {D_\delta}.
\end{cases}
\end{equation}

\begin{lem}\label{lem:a1}
For $k\ll 1$, the operator $\mathcal {S}_{D_\delta}^{k}:\ \mathcal {H}^*(\partial D_\delta)\rightarrow\mathcal {H}(\partial D_\delta)$ is invertible. Furthermore,
it holds
\begin{align}\label{singinv}
\left(\mathcal {S}_{D_\delta}^{k}\right)^{-1}=&\mathcal{L}_{D_\delta}+\mathcal{U}_k-k^{2}\ln k\mathcal{L}_{D_\delta}\mathcal{S}_{D_\delta,1}\mathcal{L}_{D_\delta}- k^{2}\left(\mathcal{L}_{D_\delta}\mathcal{S}_{D_\delta,2}\mathcal{L}_{D_\delta}+\ln k\left(\mathcal{U}_k\mathcal{S}_{D_\delta,1}\mathcal{L}_{D_\delta}+\mathcal{L}_{D_\delta}\mathcal{S}_{D_\delta,1}\mathcal{U}_k\right)\right)\nonumber\\
&+\mathcal{O}(k^2\ln^{-1} k),
\end{align}
where $\mathcal {L}_{D_\delta}=\mathcal {P}_{\mathcal {H}_0^*(\partial
D_\delta)}(\widetilde{\mathcal {S}}_{D_\delta}^{0})^{-1}$ and $\mathcal
{U}_k=\frac{\langle(\widetilde{\mathcal
{S}}_{D_\delta}^{0})^{-1}[\cdot],\varphi_0\rangle}{\mathcal {S}_{D_\delta}^{0}[\varphi_0]+\frac{\mathrm{i}}{4} c_ka_{0,\delta}}\varphi_0$. In
particular, $\mathcal {U}_k=\mathcal {O}((\ln k)^{-1})$.
\end{lem}
\begin{proof}
We set $\widehat{\mathcal
{S}}_{D_\delta}^{k}[\psi]=\Scal_{D_\delta}^0[\psi] -\frac{\mathrm{i}}{4} c_k \la 1, \psi\ra_{L^2(\p D_\delta)}$, $\forall\psi\in\mathcal {H}^*(\partial D_\delta)$. Firstly, we prove that, for $k$ small enough, $\widehat{\mathcal
{S}}_{D_\delta}^{k}:\ \mathcal {H}^*(\partial D_\delta)\rightarrow\mathcal
{H}(\partial D_\delta)$ is invertible. In fact, for any $\psi\in\mathcal {H}^*(\partial D_\delta)$
\begin{align*}
(\mathcal {S}_{D_\delta}^{0}-\widetilde{\mathcal
{S}}_{D_\delta}^{0})[\psi]=&(\mathcal {S}_{D_\delta}^{0}-\widetilde{\mathcal
{S}}_{D_\delta}^{0})[\mathcal {P}_{\mathcal {H}_0^*(\partial
D_\delta)}[\psi]+\langle\psi,\varphi_0\rangle\varphi_0]\\
=&\langle\psi,\varphi_0\rangle(\mathcal
{S}_{D_\delta}^{0}[\varphi_0]-\widetilde{\mathcal
{S}}_{D_\delta}^{0}[\varphi_0])\\
=&\langle\psi,\varphi_0\rangle(\mathcal
{S}_{D_\delta}^{0}[\varphi_0]-\chi(\partial D_\delta)),
\end{align*}
and it deduces
\begin{align}\label{s1}
\widehat{\mathcal {S}}_{D_\delta}^{k}[\psi]=&\widetilde{\mathcal
{S}}_{D_\delta}^{0}[\psi]+\langle\psi,\varphi_0\rangle(\mathcal {S}_{D_\delta}^{0}[\varphi_0]-\chi(\partial
D_\delta))\nonumber\\
&-\frac{\mathrm{i}}{4} c_k\int_{\partial
D_\delta}\left(\mathcal {P}_{\mathcal {H}_0^*(\partial
D_\delta)}[\psi]+\langle\psi,\varphi_0\rangle\varphi_0\right)
d\sigma(y)\nonumber\\
=&\widetilde{\mathcal
{S}}_{D_\delta}^{0}[\psi]+\Upsilon_k[\psi],
\end{align}
where $$\Upsilon_k[\psi]=\langle\psi,\varphi_0\rangle\left(\mathcal {S}_{D_\delta}^{0}[\varphi_0]-\chi(\partial
D_\delta)+\frac{\mathrm{i}}{4} c_ka_{0,\delta}\right).$$
Owing to the invertibility of $\widetilde{\mathcal {S}}_{D_\delta}^{0}$, we see $\widehat{\mathcal {S}}_{D_\delta}^{k}(\widetilde{\mathcal
{S}}_{D_\delta}^{0})^{-1}=\mathcal {I}+\Upsilon_k(\widetilde{\mathcal {S}}_{D_\delta}^{0})^{-1}$. From the compactness of $\Upsilon_k$ and the Fredholm alternative theorem, we only need to prove the injectivity of $\mathcal {I}+\Upsilon_k(\widetilde{\mathcal {S}}_{D_\delta}^{0})^{-1}$.

In fact, if $v\in H^{\frac{1}{2}}(\partial D_\delta)$ satisfies $(\mathcal {I}+\Upsilon_k(\widetilde{\mathcal {S}}_{D_\delta}^{0})^{-1})[v]=0$. By the definition of $\widetilde{\mathcal {S}}_{D_\delta}^{0}$ and $\Upsilon_k$, if
$(\widetilde{\mathcal {S}}_{D_\delta}^{0})^{-1}[v]\in H_0^{-\frac{1}{2}}(\partial D_\delta)$, it implies $(\mathcal {I}+\Upsilon_k(\widetilde{\mathcal {S}}_{D_\delta}^{0})^{-1})[v]=v$, and then $v=0$. If $(\widetilde{\mathcal {S}}_{D_\delta}^{0})^{-1}[v]\in \{\mu\varphi_0,\ \mu\in\mathbb{C}\}$, we find
\begin{align*}
(\mathcal {I}+\Upsilon_k(\widetilde{\mathcal {S}}_{D_\delta}^{0})^{-1})[v]
=&v+\mu\left(\mathcal {S}_{D_\delta}^{0}[\varphi_0]-\chi(\partial D_\delta)+\frac{\mathrm{i}}{4} c_ka_{0,\delta}\right)\\
=&\mu\left(\mathcal {S}_{D_\delta}^{0}[\varphi_0]+\frac{\mathrm{i}}{4} c_ka_{0,\delta}\right),
\end{align*}
since we can always find a small enough $k$ such that $\mathcal {S}_{D_\delta}^{0}[\varphi_0]\neq-\frac{\mathrm{i}}{4} c_ka_{0,\delta}$, it follows that $\mu=0$ and then $v=0$.

Since, for $k$ small enough, $\widehat{\mathcal {S}}_{D_\delta}^{k}-\mathcal {S}_{D_\delta}^{k}$ is a compact operator and $\widehat{\mathcal {S}}_{D_\delta}^{k}$
is invertible. Furthermore, it is easy to prove that $\mathcal {S}_{D_\delta}^{k}$ is injective for $k$ small enough. In fact, we consider $\psi\in H^{-\frac{1}{2}}(\partial D_\delta)$ such that $\mathcal {S}_{D_\delta}^{k}[\psi]=0$. Since $u=\mathcal {S}_{D_\delta}^{k}[\psi]$ satisfies Helmholtz equation $\Delta u+k^2u=0$ in $D_\delta$ and $\mathbb{R}^2\setminus\overline{D_\delta}$. Therefore, if $k$ is sufficiently small such that $k^2$ is neither an eigenvalue of $-\Delta$ in $D_\delta$ with the Dirichlet boundary condition on $\partial D_\delta$ nor in $\mathbb{R}^2\setminus\overline{D_\delta}$ with the Dirichlet boundary condition on $\partial D_\delta$ and the Sommerfeld radiation condition. It follows that $u=0$ and thus, $\psi=\frac{\partial u}{\partial\nu}\big|_+-\frac{\partial u}{\partial\nu}\big|_-=0$, as desired. By using the Fredholm alternative theorem, we see that, as $k$ is small enough, $\mathcal {S}_{D_\delta}^{k}$ is invertible.

Next, we verify (\ref{singinv}). Obviously, (\ref{eq:asymplayer01}) can be written as
\begin{align*}
\mathcal {S}_{D_\delta}^{k}=\widehat{\mathcal {S}}_{D_\delta}^{k}+\mathcal
{G}_k,
\end{align*}
where $\mathcal {G}_k=(k^{2}\ln k)\mathcal
{S}_{D_\delta,1}+k^{2}\mathcal {S}_{D_\delta,2}+\mathcal
{O}(k^{4}\ln k)$. Since $\mathcal {S}_{D_\delta}^{k}$ is invertible, it follows that
\begin{align*}
(\mathcal {S}_{D_\delta}^{k})^{-1}=(\mathcal {I}+(\widehat{\mathcal
{S}}_{D_\delta}^{k})^{-1}\mathcal {G}_k)(\widehat{\mathcal
{S}}_{D_\delta}^{k})^{-1}.
\end{align*}
Noting that $\|(\widehat{\mathcal {S}}_{D_\delta}^{k})^{-1}\|_{\mathcal
{L}(\mathcal {H}(\partial D_\delta),\mathcal {H}^*(\partial D_\delta))}$ is
bounded for every $k$. Thereby, when $k$ is small enough, we obtain
\begin{align}\label{Sinv}
(\mathcal {S}_{D_\delta}^{k})^{-1}=(\widehat{\mathcal
{S}}_{D_\delta}^{k})^{-1}-(\widehat{\mathcal
{S}}_{D_\delta}^{k})^{-1}\mathcal {G}_k(\widehat{\mathcal
{S}}_{D_\delta}^{k})^{-1}+\mathcal {O}(k^{4}\ln^2 k).
\end{align}
Moreover, set $\Lambda_k=(\widetilde{\mathcal
{S}}_{D_\delta}^{0})^{-1}\widehat{\mathcal {S}}_{D_\delta}^{k}$, then
\begin{align*}
\Lambda_k=&(\widetilde{\mathcal
{S}}_{D_\delta}^{0})^{-1}(\widetilde{\mathcal
{S}}_{D_\delta}^{0}+\Upsilon_k)=(\widetilde{\mathcal
{S}}_{D_\delta}^{0})^{-1}(\widetilde{\mathcal
{S}}_{D_\delta}^{0}+\langle\cdot,\varphi_0\rangle(\mathcal {S}_{D_\delta}^{0}[\varphi_0]-\chi(\partial D_\delta)+\frac{\mathrm{i}}{4} c_ka_{0,\delta}))\\
=&\mathcal {I}+\langle\cdot,\varphi_0\rangle(\mathcal {S}_{D_\delta}^{0}[\varphi_0]-\chi(\partial
D_\delta)+\frac{\mathrm{i}}{4} c_ka_{0,\delta})\varphi_0.
\end{align*}
And then $$(\Lambda_k)^{-1}=\mathcal
{I}-\langle\cdot,\varphi_0\rangle\frac{\mathcal
{S}_{D_\delta}^{0}[\varphi_0]-\chi(\partial D_\delta)+\frac{\mathrm{i}}{4} c_ka_{0,\delta}}{\mathcal
{S}_{D_\delta}^{0}[\varphi_0]+\frac{\mathrm{i}}{4} c_ka_{0,\delta}}\varphi_0.$$

Therefore, we see that
\begin{align}\label{Sinv01}
\left(\widehat{\mathcal
{S}}_{D_\delta}^{k}\right)^{-1}=(\Lambda_k)^{-1}(\widetilde{\mathcal
{S}}_{D_\delta}^{0})^{-1}=(\widetilde{\mathcal
{S}}_{D_\delta}^{0})^{-1}-\langle(\widetilde{\mathcal
{S}}_{D_\delta}^{0})^{-1}[\cdot],\varphi_0\rangle\varphi_0+\frac{\langle(\widetilde{\mathcal
{S}}_{D_\delta}^{0})^{-1}[\cdot],\varphi_0\rangle}{\mathcal {S}_{D_\delta}^{0}[\varphi_0]+\frac{\mathrm{i}}{4} c_ka_{0,\delta}}\varphi_0,
\end{align}
then, substituting (\ref{Sinv01}) into (\ref{Sinv}), it deduces (\ref{singinv}).
Furthermore, from the definition of $c_k$, $\mathcal {U}_k=\mathcal {O}((\ln k)^{-1})$ holds obviously.
\end{proof}

\begin{rem}
It is easy to find the operator $(\widetilde{\mathcal{S}}_{D_\delta}^{0})^{-1}$ is uniformly bounded with respect to $\delta$. In fact, from the definition of norm $\mathcal{H}(\partial D_\delta)$ and $\mathcal{H}^*(\partial D_\delta)$, we have
\begin{align*}
\|(\widetilde{\mathcal{S}}_{D_\delta}^{0})^{-1}[\psi]\|^2=-\left((\widetilde{\mathcal{S}}_{D_\delta}^{0})^{-1}[\psi],\psi\right)_{-\frac{1}{2},\frac{1}{2}}
=\|\psi\|_{\mathcal{H}(\partial D_\delta)}^2.
\end{align*}
Thus, we write $(\widetilde{\mathcal{S}}_{D_\delta}^{0})^{-1}$ as $(\widetilde{\mathcal{S}}_{D}^{0})^{-1}$ occasionally in what follows.
\end{rem}

Thanks to the invertibility of $\mathcal{S}_{D_\delta}^{k}$, together with the first equation in (\ref{integ01}), one can directly obtain that
\begin{equation}\label{phi}
\varphi=\left(\mathcal {S}_{D_\delta}^{k_c}\right)^{-1}\left(\mathcal {S}_{D_\delta}^{\omega}[\psi]+u^i\right).
\end{equation}
Then, from the second equation in (\ref{integ01}), we have that
\begin{equation}\label{operequ01}
\mathcal {A}_{D_\delta}(\omega)[\psi]=f,
\end{equation}
where
\begin{align}\label{operleft}
\mathcal {A}_{D_\delta}(\omega)=&\left(\frac{1}{2}\mathcal
{I}+(\mathcal {K}_{D_\delta}^{\omega})^*\right)
+\frac{1}{\varepsilon_c}\left(\frac{1}{2}\mathcal {I}-(\mathcal
{K}_{D_\delta}^{k_c})^*\right)(\mathcal {S}_{D_\delta}^{k_c})^{-1}\mathcal
{S}_{D_\delta}^{\omega},\\
f=&-\frac{\partial u^i}{\partial\nu}-\frac{1}{\varepsilon_c}\left(\frac{1}{2}\mathcal {I}-(\mathcal
{K}_{D_\delta}^{k_c})^*\right)(\mathcal {S}_{D_\delta}^{k_c})^{-1}[u^i].\label{operright}
\end{align}
Clearly,
\begin{align}
\mathcal {A}_{D_\delta}(0)=\mathcal {A}_{D_\delta,0}=&\left(\frac{1}{2}\mathcal {I}+(\mathcal
{K}_{D_\delta}^0)^*\right) +\frac{1}{\varepsilon_c}\left(\frac{1}{2}\mathcal
{I}-(\mathcal
{K}_{D_\delta}^0)^*\right)\nonumber\\
=&\frac{1}{2}\left(1+\frac{1}{\varepsilon_c}\right)\mathcal
{I}+\left(1-\frac{1}{\varepsilon_c}\right)(\mathcal
{K}_{D_\delta}^0)^*.\label{mainoper}
\end{align}

Denote by $\left\{\lambda_{j,\delta},\varphi_{j,\delta}\right\},\ j=0,1,2,\cdots$, the eigenvalue
and eigenfunction pair of $(\mathcal {K}_{D_\delta}^{0})^*$,
owing to (\ref{specexpan01}) and (\ref{mainoper}), we have
\begin{equation}\label{}
\mathcal
{A}_{D_\delta,0}[\psi]=\sum_{j=0}^\infty\tau_{j,\delta}\frac{\langle\psi,\varphi_{j,\delta}\rangle}
{\langle\varphi_{j,\delta},\varphi_{j,\delta}\rangle}\varphi_{j,\delta},
\end{equation}
where
\begin{equation}\label{taoj}
\tau_{j,\delta}=\frac{1}{2}\left(1+\frac{1}{\varepsilon_c}\right)+\left(1-\frac{1}{\varepsilon_c}\right)\lambda_{j,\delta}.
\end{equation}

First, we give the the asymptotic expansion formula of $\mathcal {A}_{D_\delta}(\omega)$ as $\omega\rightarrow0$ and $\delta\rightarrow0$.

\begin{lem}\label{lem3.4}
The operator $\mathcal {A}_{D_\delta}(\omega):\ \mathcal {H}^*(\partial
D_\delta)\rightarrow\mathcal {H}^*(\partial D_\delta)$ has the following expansion
formula:
\begin{align}\label{expan01}
\mathcal {A}_{D_\delta}(\omega)=\mathcal {A}_{D_\delta,0}+\omega^{2}\ln\omega\mathcal
{A}_{D_\delta,1}+\mathcal{O}(\omega^2).
\end{align}
Furthermore,
\begin{align}\label{expan02}
\mathcal {A}_{D_\delta}(\omega)=\mathcal {A}_{D,0}+\omega^{2}\ln\omega\mathcal
{A}_{D,1}+o(\omega^2\ln\omega)+o(1)_\delta,
\end{align}
where
\begin{align}\label{3.19}
&\mathcal
{A}_{D_\delta,0}=\frac{1}{2}\left(1+\frac{1}{\varepsilon_c}\right)\mathcal
{I}+\left(1-\frac{1}{\varepsilon_c}\right)(\mathcal {K}_{D_\delta}^0)^*;\\
&\mathcal
{A}_{D,0}=\frac{1}{2}\left(1+\frac{1}{\varepsilon_c}\right)\mathcal
{I}+\left(1-\frac{1}{\varepsilon_c}\right)\mathcal
{K}_0;\\
&\mathcal {A}_{D_\delta,1}=\frac{1}{\varepsilon_c}\left(\frac{1}{2}\mathcal
{I}-(\mathcal{K}_{D_\delta}^0)^*\right)(\widetilde{\mathcal
{S}}_{D_\delta}^{0})^{-1}\mathcal
{S}_{D_\delta,1}\left(\mathcal
{I}-\varepsilon_c\mathcal {P}_{\mathcal {H}_0^*(\partial
D_\delta)}\right)\nonumber\\
&\ \ \ \ \ \ \ +\mathcal {K}_{D_\delta,1}(\mathcal
{I}-\varepsilon_c\mathcal {P}_{\mathcal {H}_0^*(\partial D_\delta)}),\\
&\mathcal {A}_{D,1}=\frac{1}{\varepsilon_c}\left(\frac{1}{2}\mathcal
{I}-\mathcal {K}_{0}\right)(\widetilde{\mathcal
{S}}_{D}^{0})^{-1}\mathcal
{S}_{D,1,0}\left(\mathcal
{I}-\varepsilon_c\mathcal {P}_{\mathcal {H}_0^*(\partial
D)}\right)\nonumber\\
&\ \ \ \ \ \ \ +\mathcal {K}_{D,1,0}(\mathcal
{I}-\varepsilon_c\mathcal {P}_{\mathcal {H}_0^*(\partial D)}),
\end{align}
and $o(1)_\delta$ denotes an infinitesimal with respect to $\delta$.
\end{lem}
\begin{proof}
From (\ref{eq:asymplayer01}), (\ref{singinv}) and (\ref{s1}), we have that
\begin{align*}
&\mathcal {S}_{D_\delta}^{\omega}=\widetilde{\mathcal
{S}}_{D_\delta}^{0}+\Upsilon_k+(\omega^{2}\ln \omega)\mathcal{S}_{D_\delta,1}+\mathcal {O}(\omega^2),\\
&(\mathcal {S}_{D_\delta}^{k_c})^{-1}=\mathcal {L}_{D_\delta}+\mathcal
{U}_{k_c}-(\omega^{2}\ln\omega)\varepsilon_c\mathcal {L}_{D_\delta}\mathcal
{S}_{D_\delta,1}\mathcal {L}_{D_\delta}+\mathcal {O}(\omega^2).
\end{align*}
Noting the definition of $\Upsilon_k$ in (\ref{s1}), we see $\mathcal {L}_{D_\delta}\Upsilon_k=0$,
and then
\begin{align*}
(\mathcal {S}_{D_\delta}^{k_c})^{-1}\mathcal {S}_{D_\delta}^{\omega}=&\mathcal {P}_{\mathcal {H}_0^*(\partial
D_\delta)}+\mathcal {U}_{k_c}\widetilde{\mathcal {S}}_{D_\delta}^{0}+\mathcal {U}_{k_c}\Upsilon_{k_c}\\
&\ \ \ +(\omega^{2}\ln \omega)\mathcal {L}_{D_\delta}\mathcal
{S}_{D_\delta,1}\left(\mathcal {I}-\varepsilon_c\mathcal {P}_{\mathcal {H}_0^*(\partial
D_\delta)}\right)+\mathcal {O}(\omega^2).
\end{align*}
By using (\ref{re01}), it yields $\left(-\frac{1}{2}\mathcal {I}+(\mathcal
{K}_{D_\delta}^0)^*\right)\mathcal {U}_{k_c}=0$. Since $-\frac{1}{2}\mathcal {I}+(\mathcal
{K}_{D_\delta}^k)^*=\left(-\frac{1}{2}\mathcal {I}+(\mathcal
{K}_{D_\delta}^0)^*\right)+(k^{2}\ln k)\mathcal
{K}_{D_\delta,1}+\mathcal {O}(k^2)$, it deduces (\ref{expan01}). Furthermore, from Lemma \ref{lem3.2} and Lemma \ref{le:app0101}, we substitute the corresponding expansion back into (\ref{expan01}) and obtain (\ref{expan02}).
\end{proof}


Second, we have the following asymptotic expansion of $f$ given by (\ref{operright}) with respect to $\omega$ and $\delta$.

\begin{lem}\label{le:0103}
Let $f$ be defined by (\ref{operright}), we have
\beq\label{eq:eigenasy00}
f=-\mathrm{i}\omega(1-\varepsilon_c^{-1})d\cdot\nu(\Bx)+\mathcal{O}(\omega^2\ln\omega).
\eeq
\end{lem}
\begin{proof}
From (\ref{operright}), for $\Bx\in\partial D_\delta$, and using the expansion
$e^{i\omega d\cdot \Bx}=1+\mathrm{i}\omega d\cdot \Bx+\mathcal{O}(\omega^2)$, we have
\begin{align*}
f(x)
&=-\omega \mathrm{i}d\cdot\nu(\Bx)+\mathcal{O}(\omega^2)\\
&\ \ \ -\varepsilon_c^{-1}\left(\left(\frac{1}{2}\mathcal{I}
-(\mathcal{K}_{D_\delta}^0)^*\right)-\omega^2\varepsilon_c\ln(\omega\sqrt{\varepsilon_c})
\mathcal{K}_{D_\delta,1}+\mathcal{O}(\omega^2\varepsilon_c)\right)(\mathcal {L}_{D_\delta}+\mathcal{U}_{k_c}\\
&\ \ \ \ \ \ -(\omega^{2}\ln\omega)\varepsilon_c\mathcal {L}_{D_\delta}\mathcal
{S}_{D_\delta,1}\mathcal {L}_{D_\delta}+\mathcal {O}(\omega^2))[1+\mathrm{i}\omega d\cdot \Bx+\mathcal{O}(\omega^2)]\\
&=-\omega \mathrm{i}d\cdot\nu(x)-\varepsilon_c^{-1}\left(\frac{1}{2}\mathcal{I}
-(\mathcal{K}_{D_\delta}^0)^*\right)\bigg(\mathcal{L}_{D_\delta}[1]+\mathcal{U}_{k_c}[1]\\
&\ \ \ \ \ +\mathcal{L}_{D_\delta}[\mathrm{i}\omega d\cdot \Bx]+\mathcal{U}_{k_c}[\mathrm{i}\omega d\cdot \Bx]\bigg)+\mathcal{O}(\omega^2\ln\omega).
\end{align*}
Since $\left(\frac{1}{2}\mathcal{I}-(\mathcal{K}_{D_\delta}^0)^*\right)
(\widetilde{\mathcal{S}}_{D_\delta}^{0})^{-1}[1]=0$ and $\left(\frac{1}{2}\mathcal{I}-(\mathcal{K}_{D_\delta}^0)^*\right)\mathcal{U}_{k_c}=0$, by applying Lemma \ref{lem3.2}, we get that
$$f=-\mathrm{i}\omega\left(d\cdot\nu(\Bx)+\varepsilon_c^{-1}\left(\frac{1}{2}\mathcal {I}-(\mathcal{K}_{D_\delta}^0)^*\right)(\widetilde{\mathcal{S}}_{D_\delta}^{0})^{-1}[d\cdot \Bx]\right)+\mathcal{O}(\omega^2\ln\omega).$$
Note that
$$
\left(\frac{1}{2}\mathcal {I}-(\mathcal{K}_{D_\delta}^0)^*\right)(\widetilde{\mathcal{S}}_{D_\delta}^{0})^{-1}[d\cdot \Bx]=-d\cdot \nu,
$$
we obtain the result \eqnref{eq:eigenasy00}.
\end{proof}

From the operator equation (\ref{operequ01}), using Lemma \ref{lem3.4} and Lemma \ref{le:0103}, there holds
\beq\label{opeq}
\left(\lambda(\varepsilon_c)\mathcal{I}-\mathcal{K}_{0}-\delta\mathcal{K}_{1}+\Ocal(\delta^2)\right)[\psi]=\omega f_1+\mathcal{O}(\omega^2\ln\omega),
\eeq
where $f_1=\mathrm{i}d\cdot\nu$.
Before proceeding, we need the spectral result of $\Kcal_0$. Define the operator $A_\delta$ by
\beq\label{eq:defAop01}
A_\delta[\psi](x_1):=\frac{1}{\pi}\int_{-L/2}^{L/2}\frac{\delta}{(x_1-y_1)^2+4\delta^2}\psi(y_1)dy_1, \quad \psi\in L^2(-L/2, L/2).
\eeq
We shall see that the operator $A_\delta$ is strongly related to the operator $\Kcal_0$.
The following result on the spectral of $A_\delta$ is given in \cite{fang2021}:
\begin{lem}[Lemma 3.2 in \cite{fang2021}]\label{le:Adel01}
Suppose $A_\delta$ is defined in \eqnref{eq:defAop01}, then it holds that
\beq\label{eq:eigenA01}
A_\delta[y_1^n](x_1)=
\frac{1}{2}x_1^n + o(1),  \quad \Bx\in \Gamma_{j}\setminus\big(\iota_{\delta}(P)\cup \iota_{\delta}(Q)\big), \quad n \geq 0.\\
\eeq
\end{lem}
\begin{rem}
By following the proof in \cite{fang2021}, together with asymptotic analysis, one can get more accurate result than \eqnref{eq:eigenA01}, which is
\beq\label{eq:eigenA02}
\begin{split}
A_\delta[y_1^n](x_1)=&
\frac{1}{2}x_1^n + \chi(\iota_{\delta^\epsilon}(P)\cup \iota_{\delta^\epsilon}(Q))\Ocal(\delta^{1-\epsilon})+\Ocal(\delta),  \quad 0<\epsilon<1, \\
  &\quad\quad\quad\quad \Bx\in \Gamma_{j}\setminus\big(\iota_{\delta}(P)\cup \iota_{\delta}(Q)\big), \quad n \geq 0.
\end{split}
\eeq
\end{rem}
With the above results on hand, we can show the following asymptotic result for the density $\psi$.

\begin{lem}\label{le:main01}
Suppose $\psi$ is the solution of \eqnref{operequ01}, then one has
\beq\label{eq:asymvphi01}
\psi(\Bx)=\left\{
\begin{array}{l}
(-1)^{j}\omega\mathrm{i}d_2\left(\lambda(\varepsilon_c)I+A_\delta\right)^{-1}[1] +\omega\chi(\iota_{\delta^\epsilon}(P)\cup \iota_{\delta^\epsilon}(Q))\Ocal(\delta^{2(1-\epsilon)})\\
 \quad\quad\quad\quad\quad\quad +\Ocal(\omega\delta^2)+\mathcal{O}(\omega^2\ln\omega), \quad \quad\quad\quad\quad\Bx\in \Gamma_{j}\setminus(\iota_{\delta}(P)\cup \iota_{\delta}(Q)),\medskip \\
\omega\mathrm{i}(\lambda(\varepsilon_c) \mathcal{I} -\Kcal_1^*)^{-1}\left[d\cdot\nu(\Bx)\right]+\omega\cdot o(1)_\delta+\mathcal{O}(\omega^2\ln\omega), \quad\quad \Bx\in S_\delta^a\cup \iota_{\delta}(P), \medskip\\
\omega\mathrm{i}(\lambda(\varepsilon_c) \mathcal{I} -\Kcal_2^*)^{-1}\left[d\cdot\nu(\Bx)\right]+\omega\cdot o(1)_\delta+\mathcal{O}(\omega^2\ln\omega), \quad\quad \Bx\in S_\delta^b\cup \iota_{\delta}(Q),
\end{array}
\right.
\eeq
where the operators $\Kcal_1^*$ and $\Kcal_2^*$ are defined by
\beq\label{eq:remtm01}
\begin{split}
\Kcal_1^*[\varphi_1](\Bx):=&\int_{S_\delta^a\cup \iota_{\delta}(P)}\frac{\la\Bx-\By,\nu_x\ra}{|\Bx-\By|^2}\varphi_1(\By)d\sigma(\By),\\ 
\Kcal_2^*[\varphi_2](\Bx):=&\int_{S_\delta^b\cup \iota_{\delta}(Q)}\frac{\la\Bx-\By,\nu_x\ra}{|\Bx-\By|^2}\varphi_2(\By)d\sigma(\By),
\end{split}
\eeq
respectively.
\end{lem}
\begin{proof}
From (\ref{opeq}), by using \eqnref{eq:leasyNP02}, \eqnref{eq:leasyNP01} and \eqnref{eq:descrK101},
one can show
\beq\label{eq:bdint1202}
\begin{split}
&\lambda(\varepsilon_c)\psi(x_1,-\delta)-\frac{1}{\pi}\int_{-L/2}^{L/2}\frac{\delta}{(x_1-y_1)^2+4\delta^2}\psi(y_1,\delta)dy_1\\
&+\Big(\chi(\iota_{\delta^\epsilon}(P)\cup \iota_{\delta^\epsilon}(Q))\Ocal(\delta^{2(1-\epsilon)}) +\Ocal(\delta^2)\Big)\psi(x_1,\delta)\\
=&\omega\mathrm{i}d_2+\mathcal{O}(\omega^2\ln\omega),\ \ \ \quad 0<\epsilon<1, \quad \Bx\in \Gamma_{1}\setminus \big(\iota_{\delta}(P)\cup \iota_{\delta}(Q)\big),
\end{split}
\eeq
and
\beq\label{eq:bdint120202}
\begin{split}
&\lambda(\varepsilon_c)\psi(x_1,\delta)-\frac{1}{\pi}\int_{-L/2}^{L/2}\frac{\delta}{(x_1-y_1)^2+4\delta^2}\psi(y_1,-\delta)dy_1\\
&+\Big(\chi(\iota_{\delta^\epsilon}(P)\cup \iota_{\delta^\epsilon}(Q))\Ocal(\delta^{2(1-\epsilon)}) +\Ocal(\delta^2)\Big)\psi(x_1,-\delta)\\
=&-\omega\mathrm{i}d_2+\mathcal{O}(\omega^2\ln\omega),\ \ \ \quad 0<\epsilon<1, \quad \Bx\in \Gamma_{2}\setminus \big(\iota_{\delta}(P)\cup \iota_{\delta}(Q)\big),
\end{split}
\eeq
One can thus decompose the density $\psi$ into
\beq\label{eq:decompsi01}
\begin{split}
\psi(x_1,(-1)^j\delta)=&\omega\Big(\psi^{(0)}(x_1,(-1)^j\delta)\\
&\quad+\delta^{2(1-\epsilon)}\chi(\iota_{\delta^\epsilon}(P)\cup \iota_{\delta^\epsilon}(Q))\psi^{(1)}(x_1,(-1)^j\delta)+\Ocal(\delta^2)\Big)+\Ocal(\omega^2\ln\omega),
\end{split}
\eeq
for $|x_1|\leq L/2-\Ocal(\delta)$.
Furthermore, from \eqnref{eq:bdint1202} and \eqnref{eq:bdint120202}, one has
\beq\label{eq:relapsi01}
\psi^{(0)}(x_1,-\delta)=-\psi^{(0)}(x_1,\delta), \quad |x_1|\leq L/2-\Ocal(\delta).
\eeq
Thus one can derive that
\beq\label{psisum01}
\begin{split}
\psi(x_1,-\delta)=&\omega\mathrm{i}d_2\left(\lambda(\varepsilon_c)I+A_\delta\right)^{-1}[1]
+\omega\chi(\iota_{\delta^\epsilon}(P)\cup \iota_{\delta^\epsilon}(Q))\Ocal(\delta^{2(1-\epsilon)}) \\
& +\Ocal(\omega\delta^2) +\mathcal{O}(\omega^2\ln\omega),\ \ \ \ \ \text{for}\ |x_1|\leq L/2-\Ocal(\delta).
\end{split}
\eeq
Similarly, we obtain
\beq\label{psi01}
\begin{split}
\psi(x_1,\delta)=&-\omega\mathrm{i}d_2\left(\lambda(\varepsilon_c)I+A_\delta\right)^{-1}[1]
 +\omega\chi(\iota_{\delta^\epsilon}(P)\cup \iota_{\delta^\epsilon}(Q))\Ocal(\delta^{2(1-\epsilon)}) \\
& +\Ocal(\omega\delta^2) +\mathcal{O}(\omega^2\ln\omega),\ \ \ \ \ \text{for}\ |x_1|\leq L/2-\Ocal(\delta).
\end{split}
\eeq
Thus one obtains the first equation in \eqnref{eq:asymvphi01}.
For $\Bx\in S^a\cup \iota_\delta(P)$, by making use of \eqnref{opeq}, \eqnref{eq:eigenA01}, together with the first equation in \eqnref{eq:asymvphi01} one has
\begin{align*}
(\lambda(\varepsilon_c)\mathcal{I} -\Kcal_1^*)[\varphi](\Bx)&= \omega\mathrm{i}d\cdot\nu(\Bx)+\omega\cdot o(1)_\delta+\mathcal{O}(\omega^2\ln\omega), \quad \Bx\in S_\delta^a\cup\iota_\delta(P).
\end{align*}
In a similar manner, one can show that
\begin{align*}
(\lambda(\varepsilon_c)\mathcal{I} -\Kcal_2^*)[\varphi](\Bx)&= \omega\mathrm{i}d\cdot\nu(\Bx) +\omega\cdot o(1)_\delta+\mathcal{O}(\omega^2\ln\omega), \quad \Bx\in S_\delta^b\cup\iota_\delta(Q).
\end{align*}
and so the last two equations in \eqnref{eq:asymvphi01} follows.
\end{proof}
We are now in the position of proving the first main result Theorem \ref{eq:thmanin01}.
\begin{proof}[Proof of Theorem \ref{eq:thmanin01}]
For the sake of simplicity we use the notation $\Gl$ and omit its dependence. By using \eqnref{eq:asyHank01}, \eqnref{solu01} and Taylor's expansion along with $\Gamma_0$, we have that
\beq\label{eq:thpf01}
\begin{split}
u(\Bx)=&u^i(\Bx)+ \int_{S_\delta^f\setminus(\iota_{\delta}(P)\cup \iota_{\delta}(Q))}G_0(\Bx-\Bz_y)\psi(\By)\,d\sigma(\By) \\
&\quad \quad \ \,+ \delta \int_{S_\delta^f\setminus(\iota_{\delta}(P)\cup \iota_{\delta}(Q))}\nabla_\By G_0(\Bx-\Bz_y)\cdot \nu_\By \psi(\By)\,d\sigma(\By)\\
&\quad \quad \ \, + \int_{S_\delta^a\cup \iota_\delta(P)} G_0(\Bx-\Bz_y) \psi(\By)\,d\sigma(\By)+ \int_{S_\delta^b\cup \iota_\delta(Q)} G_0(\Bx-\Bz_y) \psi(\By)\,d\sigma(\By)+\Ocal(\delta^2\omega^2\ln\omega).
\end{split}
\eeq
Furthermore, by substituting the asymptotic expansion of $\psi$ in \eqnref{eq:asymvphi01} into above equation, and using the fact that (see Lemma 3.3 in \cite{fang2021})
\[
\begin{split}
\int_{S_\delta^a\cup \iota_\delta(P)}  \psi(\By)\,d\sigma(\By)=&-2\mathrm{i}\omega\delta\Big(\lambda(\varepsilon_c)-\frac{1}{2}\Big)^{-1} d_1+\omega o(\delta), \\
\int_{S_\delta^b\cup \iota_\delta(Q)}  \psi(\By)\,d\sigma(\By)=&2\mathrm{i}\omega\delta\Big(\lambda(\varepsilon_c)-\frac{1}{2}\Big)^{-1} d_1+\omega o(\delta),
\end{split}
\]
one can derive that
\beq\label{eq:thpf02}
\begin{split}
&\int_{S_\delta^a\cup \iota_\delta(P)} G_0(\Bx-\Bz_y) \psi(\By)\,d\sigma(\By)
=\int_{S_\delta^a\cup \iota_\delta(P)} (G_0(\Bx-P) +\Ocal(\delta))\psi(\By)\,d\sigma(\By)\\
=&-\omega\delta\frac{\mathrm{i}}{\pi}\Big(\lambda(\varepsilon_c)-\frac{1}{2}\Big)^{-1} d_1\ln|\Bx-P|+\omega o(\delta)\\
=&-\omega\delta\frac{\mathrm{i}}{2\pi}\Big(\lambda(\varepsilon_c)-\frac{1}{2}\Big)^{-1} d_1\ln((x_1-L/2)^2+x_2^2)+\omega o(\delta),
\end{split}
\eeq
and
\beq\label{eq:thpf02}
\begin{split}
&\int_{S^b\cup \iota_\delta(Q)} G_0(\Bx-\Bz_y) \psi(\By)\,d\sigma(\By)\\
=&\omega\delta\frac{\mathrm{i}}{2\pi}\Big(\lambda(\varepsilon_c)-\frac{1}{2}\Big)^{-1} d_1\ln((x_1+L/2)^2+x_2^2)+\omega o(\delta),
\end{split}
\eeq
Finally, by using \eqnref{eq:asymvphi01}, there holds
\beq\label{eq:thpf03}
\begin{split}
&\int_{S_\delta^f\setminus(\iota_{\delta}(P)\cup \iota_{\delta}(Q))} G_0(\Bx-\Bz_y) \psi(\By)\,d\sigma(\By)=\omega o(\delta),
\end{split}
\eeq
and
\beq\label{eq:thpf0301}
\begin{split}
&\int_{S_\delta^f\setminus(\iota_{\delta}(P)\cup \iota_{\delta}(Q))} \nabla_\By G_0(\Bx-\Bz_y)\cdot \nu_\By \psi(\By)\,d\sigma(\By)\\
=&\frac{\omega\mathrm{i}}{2\pi}\delta d_2\int_{-L/2}^{L/2}\frac{x_2}{(x_1-y_1)^2+x_2^2}(\lambda(\varepsilon_c) I+A_\delta)^{-1}[1](y_1)dy_1+\omega o(\delta),
\end{split}
\eeq
By combing the above results one can finally obtain \eqnref{eq:thmanin0101}.
\end{proof}

\section{Resonance analysis of the nanorod}

In this section, according to definition \ref{depr}, we proceed to analyze the plasmon resonance of the scattering system \eqref{eq:helm01}. It is worth emphasizing that we do not analyze the plasmon resonance by utilizing the expansion formula of the scattering field (\ref{eq:thmanin0101}) directly. We just use some estimation of scattering field and the relevant operators.
First, we derive the gradient estimate of the scattering field $u^s$ outside the nanorod $D_\delta$.
%

\begin{lem}\label{lem4.1}
Let $\psi=\psi_c+a_{0,\delta}^{-1}\langle\psi,\varphi_{0,\delta}\rangle\varphi_{0,\delta}$, where $\psi_c\in\mathcal{H}_0^*(\partial D)$. Then for the scattering solution of (\ref{eq:helm01}), we have the following estimate
\begin{align}\label{4.1}
\left|\|\nabla u^s\|_{L^2(\mathbb{R}^2\setminus\overline{D_\delta})}^2-\|\psi_c\|^2\right|\lesssim\omega^2|\ln\omega|^2 a_{0,\delta}^{-1}\left|\langle\psi,\varphi_{0,\delta}\rangle\right|^2.
\end{align}
\end{lem}

\begin{proof}
Suppose $B_R$ is a sufficiently large disk such that $\overline{D_\delta}\subset B_R$, and $\check{\mathcal{S}}_{D_\delta}^{k_m}[\psi]=\mathcal{S}_{D_\delta}^{k_m}[\psi]
+\frac{\mathrm{i}}{4}c_{k_m}\langle1,\psi\rangle_{\partial D_\delta}$. By ultilizing the divergence theorem in $B_R\setminus\overline{D_\delta}$, the jump relation \eqref{eq:trace} and the Sommerfeld radiation condition, we have that
\begin{align*}
\int_{B_R\setminus\overline{D_\delta}}|\nabla u^s|^2dx= &\bar{k}_m^2\int_{B_R\setminus\overline{D_\delta}}|u^s|^2dx-\int_{\partial D_\delta}u^s\overline{\frac{\partial u^s}{\partial\nu}\Big|+}d\sigma+\int_{\partial B_R}u^s\overline{\frac{\partial u^s}{\partial R}}d\sigma\\
=&\bar{k}_m^2\int_{B_R\setminus\overline{D_\delta}}|\check{\mathcal{S}}_{D_\delta}^{k_m}[\psi]|^2dx-\int_{\partial D_\delta}\check{\mathcal{S}}_{D_\delta}^{k_m}[\psi]\overline{\left(\frac{1}{2}\mathcal{I}+(\mathcal{K}_{D_\delta}^{k_m})^*\right)[\psi]}d\sigma\\
&+\int_{\partial B_R}\check{\mathcal{S}}_{D_\delta}^{k_m}[\psi]\cdot\overline{\mathrm{i}k_m\check{\mathcal{S}}_{D_\delta}^{k_m}[\psi]+\mathcal{O}(R^{-\frac{3}{2}})}d\sigma.
\end{align*}
Notice (\ref{eq:asymplayer01}) and (\ref{eq:asymplayer02}), it implies
\begin{align}
\left|\int_{\partial D_\delta}\check{\mathcal {S}}_{D_\delta}^{k_m}[\psi]\overline{\left(\frac{1}{2}\mathcal{I}+(\mathcal{K}_{D_\delta}^{k_m})^*\right)[\psi]}d\sigma\right|
\leq\left|\int_{\partial D_\delta}\mathcal {S}_{D_\delta}^0[\psi]\overline{\left(\frac{1}{2}\mathcal{I}+(\mathcal{K}_{D_\delta}^0)^*\right)[\psi]}d\sigma\right|+\left|E\right|,
\end{align}
where
\begin{align}
E=&\int_{\partial D_\delta}\mathcal{S}_{D_\delta}^0[\psi]\overline{ \omega^{2}\ln\omega\mathcal{K}_{D_\delta,1}[\psi]+\omega^2\mathcal{O}(\|\psi\|)}d\sigma\\
&\ \ \ \ \ \ \ \ +\int_{\partial D_\delta}\left(\omega^{2}\ln\omega\mathcal
{S}_{D_\delta,1}[\psi]+\omega^{2}\mathcal
{S}_{D_\delta,2}[\psi]+\omega^4\ln\omega\mathcal{O}(\|\psi\|)\right)\overline{\left(\frac{1}{2}\mathcal
{I}+(\mathcal{K}_{D_\delta}^{k_m})^*\right)[\psi]}d\sigma\nonumber.
\end{align}
Since $\omega$ is small enough, it is easy to see that $|E|\lesssim |\omega\ln\omega|^2\|\psi\|^2$.

Next, we estimate $\left|\int_{\partial D_\delta}\mathcal {S}_{D_\delta}^0[\psi]\overline{\left(\frac{1}{2}\mathcal{I}+(\mathcal{K}_{D_\delta}^0)^*\right)[\psi]}d\sigma\right|$.
Since $(\mathcal{K}_{D_\delta}^0)^*[\varphi_{0,\delta}]=\frac{1}{2}\varphi_{0,\delta}$ and $\mathcal{S}_{D_\delta}^0[\varphi_{0,\delta}]=0$, for $\psi=\psi_c+a_{0,\delta}^{-1}\langle\psi,\varphi_{0,\delta}\rangle\varphi_{0,\delta}$, it follows that
\begin{align*}
&\int_{\partial D_\delta}\mathcal {S}_{D_\delta}^0[\psi]\overline{\left(\frac{1}{2}\mathcal{I}+(\mathcal{K}_{D_\delta}^0)^*\right)[\psi]}d\sigma\\
=&\int_{\partial D_\delta}\mathcal {S}_{D_\delta}^0[\psi_c]\overline{\left(a_{0,\delta}^{-1}\langle\psi,\varphi_{0,\delta}\rangle\varphi_{0,\delta}+\left(\frac{1}{2}\mathcal{I}+(\mathcal{K}_{D_\delta}^0)^*\right)[\psi_c]\right)}d\sigma\\
=&\int_{\partial D_\delta}\mathcal{S}_{D_\delta}^0\left[\sum_{j=1}^\infty a_{j,\delta}^{-1}\langle\psi_c,\varphi_{j,\delta}\rangle\varphi_{j,\delta}\right]\overline{\left(a_{0,\delta}^{-1}\langle\psi,\varphi_{0,\delta}\rangle\varphi_{0,\delta}
+\left(\frac{1}{2}\mathcal{I}+(\mathcal{K}_{D_\delta}^0)^*\right)\left[\sum_{l=1}^\infty a_{l,\delta}^{-1}\langle\psi_c,\varphi_{l,\delta}\rangle\varphi_{l,\delta}\right]\right)}d\sigma\\
=&\int_{\partial D_\delta}\sum_{j=1}^\infty a_{j,\delta}^{-1}\langle\psi_c,\varphi_{j,\delta}\rangle\mathcal{S}_{D_\delta}^0[\varphi_{j,\delta}]\overline{\left(a_{0,\delta}^{-1}\langle\psi,\varphi_{0,\delta}\rangle\varphi_{0,\delta}
+\sum_{l=1}^\infty a_{l,\delta}^{-1}\langle\psi_c,\varphi_{l,\delta}\rangle\left(\frac{1}{2}+\lambda_{l,\delta}\right)\varphi_{l,\delta}\right)}d\sigma\\
=&\sum_{j=1}^\infty a_{0,\delta}^{-1}\overline{\langle\psi,\varphi_{0,\delta}\rangle}a_{j,\delta}^{-1}\langle\psi_c,\varphi_{j,\delta}\rangle\int_{\partial D_\delta}\mathcal{S}_{D_\delta}^0[\varphi_{j,\delta}]\overline{\varphi_{0,\delta}}d\sigma\\
&\quad +\sum_{j,l=1}^\infty\left(\frac{1}{2}+\lambda_{l,\delta}\right)a_{l,\delta}^{-1}\overline{\langle\psi_c,\varphi_{l,\delta}\rangle}a_{j,\delta}^{-1}\langle\psi_c,\varphi_{j,\delta}\rangle\int_{\partial D_\delta}\mathcal{S}_{D_\delta}^0[\varphi_{j,\delta}]\overline{\varphi_{l,\delta}}d\sigma.
\end{align*}
Noting that
\begin{align*}
\int_{\partial D_\delta}\mathcal{S}_{D_\delta}^0[\varphi_{j,\delta}]\overline{\varphi_{l,\delta}}d\sigma=-\langle\varphi_{l,\delta},\varphi_{j,\delta}\rangle=
\begin{cases}
-a_{j,\delta},\ \ \ l=j,\\
0,\ \ \ l\neq j,
\end{cases}
\end{align*}
it deduces
\begin{align*}
\int_{\partial D_\delta}\mathcal {S}_{D_\delta}^0[\psi]\overline{\left(\frac{1}{2}\mathcal{I}+(\mathcal{K}_{D_\delta}^0)^*\right)[\psi]}d\sigma
=-\sum_{j=1}^\infty\left(\frac{1}{2}+\lambda_{j,\delta}\right)a_{j,\delta}^{-1}\left|\langle\psi_c,\varphi_{j,\delta}\rangle\right|^2.
\end{align*}
Since $\lambda_{j,\delta}\in\left(-\frac{1}{2},\frac{1}{2}\right)$ $(j\geq1)$, we can get
\begin{align}\label{4.3}
\left|\int_{\partial D_\delta}\mathcal {S}_{D_\delta}^0[\psi]\overline{\left(\frac{1}{2}\mathcal{I}+(\mathcal{K}_{D_\delta}^0)^*\right)[\psi]}d\sigma
\right|\approx\|\psi_c\|^2.
\end{align}
Furthermore, by Cauchy's inequality, it finds
\begin{align}\label{4.4}
\left|\int_{\partial B_R}u^s\cdot\overline{\mathrm{i}k_mu^s+\mathcal{O}(R^{-\frac{3}{2}})}d\sigma\right|\lesssim&\omega\|u^s\|_{L^2(\partial B_R)}^2+\int_{\partial B_R}|u^s\cdot\overline{\mathcal{O}(R^{-\frac{3}{2}})}|d\sigma\nonumber\\
\lesssim&\omega\|\psi\|^2+\omega\mathcal{O}(R^{-1})\cdot\|\psi\|.
\end{align}
By combing (\ref{4.3}) and (\ref{4.4}), for $\omega\ll1$, we have
\begin{align*}\label{}
\|\nabla u^s\|_{L^2(B_R\setminus\overline{D_\delta})}^2\lesssim\|\psi_c\|^2+|\omega\ln\omega|^2 a_{0,\delta}^{-1}\left|\langle\psi,\varphi_{0,\delta}\rangle\right|^2+\omega\mathcal{O}(R^{-1})\cdot\|\psi\|.
\end{align*}
Similarly, by (\ref{4.3}) and (\ref{4.4}), we also deduce the inverse inequality as
\begin{align*}
\|\nabla u^s\|_{L^2(B_R\setminus\overline{D_\delta})}^2\gtrsim\|\psi_c\|^2-|\omega\ln\omega|^2
a_{0,\delta}^{-1}\left|\langle\psi,\varphi_{0,\delta}\rangle\right|^2-\omega\mathcal{O}(R^{-1})\cdot\|\psi\|.
\end{align*}
Thus, by letting $R\rightarrow\infty$, we see that the estimate (\ref{4.1}) holds.
\end{proof}

Before proceeding with the gradient analysis of the scattering field $u^s$ outside the nanorod $D_\delta$, we consider the parameter choice of the permittivity with an imaginary part. In fact, in real applications, nano-metal materials always contain losses, which are reflected in the imaginary part of a complex electric permittivity $\varepsilon_c$. As shall be shown, like the frequency $\omega$ and the size $\delta$, the lossy parameter also plays a key role in the plasmon resonance of the scattering field of the nanorod. Let $\theta=\Re\left(\frac{1}{\varepsilon_c}\right)$, $\rho=\Im\left(\frac{1}{\varepsilon_c}\right)<0$. Then
$\tau_{j,\delta}$ given by (\ref{taoj}) can be written as
\begin{equation}\label{4.5}
\tau_{j,\delta}=\frac{1}{2}(\theta+1)-(\theta-1)\lambda_{j,\delta}+\rho(\frac{1}{2}-\lambda_{j,\delta})\mathrm{i}.
\end{equation}
Notice that $\lambda_{j,\delta}=1/2+\mathcal{O}(\delta^{1-\epsilon})$, $(0<\epsilon<1)$, then $\tau_{j,\delta}=1+(1-\theta)\cdot \mathcal{O}(\delta^{1-\epsilon})-\mathcal{O}(\delta^{1-\epsilon})\cdot\rho\mathrm{i}$.

By considering the principal equation $\mathcal{A}_{D_\delta,0}[\psi_0]=f$, where $\mathcal{A}_{D_\delta,0}$ is defined by (\ref{mainoper}) and $\psi_0, f\in \mathcal{H}^*(\partial D_\delta)$. Then, applying the eigenfunction expansion, it follows that
\begin{align}\label{4.6}
\psi_0=\mathcal {A}_{D_\delta,0}^{-1}[f]=\sum_{j=0}^\infty\frac{a_{j,\delta}^{-1}\langle f,\varphi_{j,\delta}\rangle}{\tau_{j,\delta}}\varphi_{j,\delta}.
\end{align}

\begin{lem}\label{lem4.2}
Let $\psi_0$ be given by (\ref{4.6}) and has the decomposition $\psi_0=\psi_{0,c}+c\varphi_0$, ($\psi_{0,c}\in\mathcal{H}_0^*(\partial D_\delta)$, $c$ is a constant). Then, for sufficiently small $\delta$, $|\rho|$, and $\epsilon\in(0,1)$, it holds that
\begin{enumerate}

\item[(1)] $\|\mathcal{A}_{D_\delta,0}^{-1}\|_{\mathcal{L}(\mathcal{H}^*(\partial D_\delta),\mathcal{H}^*(\partial D_\delta))}\lesssim(|\rho|\cdot \mathcal{O}(\delta^{1-\epsilon}))^{-1}$.

\item[(2)] If $\frac{1}{2}\frac{\theta+1}{\theta-1}\neq\lambda_{j,\delta}, (\forall\  j\geq0)$, for any fixed $\delta$, then $\|\mathcal{A}_{D_\delta,0}^{-1}\|_{\mathcal{L}(\mathcal{H}^*(\partial D_\delta),\mathcal{H}^*(\partial D_\delta))}\lesssim C$ for some positive constant $C$.

\item[(3)] If $\lambda_{j_*,\delta}=\frac{1}{2}\frac{\theta+1}{\theta-1}-\rho\frac{1}{\theta-1},$ for some $j_*\geq1$, then $\|\psi_{0,c}\|\gtrsim |\rho|^{-1}a_{j_*,\delta}^{-\frac{1}{2}}\left|\langle f,\varphi_{j_*,\delta}\rangle\right|$.
\end{enumerate}
\end{lem}

\begin{proof}
(1) For $j\neq0$, since $\left|\tau_{j,\delta}^{-1}\right|\lesssim\frac{1}{|\rho|(\frac{1}{2}-\lambda_{j,\delta})}\lesssim(|\rho|\cdot \mathcal{O}(\delta^{1-\epsilon}))^{-1}$, it follows that
\begin{equation}\label{}
\|\psi_0\|^2\lesssim(|\rho|\cdot \mathcal{O}(\delta^{1-\epsilon}))^{-2}\sum_{j=0}^\infty a_{j,\delta}^{-1}\left|\langle f,\varphi_{j,\delta}\rangle\right|^2\lesssim(|\rho|\cdot \mathcal{O}(\delta^{1-\epsilon}))^{-2}\|f\|^2.
\end{equation}
Hence, $\|\mathcal{A}_{D_\delta,0}^{-1}\|_{\mathcal{L}(\mathcal{H}^*(\partial D_\delta),\mathcal{H}^*(\partial D_\delta))}\lesssim(|\rho|\cdot \mathcal{O}(\delta^{1-\epsilon}))^{-1}$.

\medskip

(2) If $\frac{1}{2}\frac{\theta+1}{\theta-1}\neq\lambda_{j,\delta}$, then, for $j\geq0$, we have $\left|\frac{1}{2}\frac{\theta+1}{\theta-1}-\lambda_{j,\delta}\right|\geq c_0$, where $c_0$
is a positive constant. Therefore, $\left|\tau_{j,\delta}^{-1}\right|\lesssim1$ and $\|\psi_0\|^2\lesssim\|f\|^2$, i.e.,
$\|\mathcal{A}_{D_\delta,0}^{-1}\|_{\mathcal{L}(\mathcal{H}^*(\partial D_\delta),\mathcal{H}^*(\partial D_\delta))}\lesssim C$.

\medskip

(3) When $\lambda_{j_*,\delta}=\frac{1}{2}\frac{\theta+1}{\theta-1}-\rho\frac{1}{\theta-1},$ for some $j_*\geq1$, by (\ref{4.5}), one has
$\tau_{j_*,\delta}=\rho (1-\mathcal{O}(\delta^{1-\epsilon})\cdot\mathrm{i})$, it then follows that
\begin{equation}\label{}
\|\psi_{0,c}\|\gtrsim a_{j_*,\delta}^{-\frac{1}{2}}\left|\langle\psi_{0,c},\varphi_{j_*,\delta}\rangle\right|\gtrsim\frac{a_{j_*,\delta}^{-\frac{1}{2}}\left|\langle f,\varphi_{j_*,\delta}\rangle\right|}{\left|\rho(1-\mathcal{O}(\delta^{1-\epsilon})\cdot\mathrm{i})\right|}\gtrsim |\rho|^{-1}a_{j_*,\delta}^{-\frac{1}{2}}\left|\langle f,\varphi_{j_*,\delta}\rangle\right|,
\end{equation}
which completes the proof.
\end{proof}


With Lemma \ref{lem4.2}, we can establish the following key result for estimating the gradient of the scattering field $u^s$, which provide resonant and non-resonant conditions for the scattering system \eqref{eq:helm01} associated with the nanorod $D_\delta$ according to criterion \eqref{prdf01} in Definition~\ref{depr}.

\begin{thm}\label{thm4.3}
Let $u^s$ be the scattering solution of (\ref{eq:helm01}). Assume that
\begin{equation}\label{codition01}
\omega^2\ln\omega(1+\mathcal{O}((\ln\omega)^{-1}))|\rho|^{-1}\leq c_1,
\end{equation}
for a sufficiently small $c_1$, then we have that

\begin{enumerate}

\item[(1)] If there holds
$\frac{1}{2}\frac{\theta+1}{\theta-1}\neq \lambda_{j,\delta}$
 for any $j\geq0$ and $\delta$, then there exists a constant $C$ independent of $\delta$ such that
\begin{equation}
\left\|\nabla u^s\right\|_{L^2(\mathbb{R}^2\setminus\overline{D_\delta})}\leq C.
\end{equation}

\item[(2)] Suppose $d_2=0$. If there holds
$\lambda_{j_*,\delta}=\frac{1}{2}\frac{\theta+1}{\theta-1}-\rho\frac{1}{\theta-1}$ and
\beq\label{eq:rntmd01}
d_1\langle\langle 1,\ \tilde{\varphi}_{j_*}\rangle\rangle_{S^f}\neq0
\eeq
 for some $j_*\geq 1$, it holds
\begin{align}\label{bl1}
\left\|\nabla u^s\right\|_{L^2(\mathbb{R}^2\setminus\overline{D_\delta})}\gtrsim|\rho|^{-1}\omega\delta
+o(\omega\delta)+\mathcal{O}(\omega^2\ln\omega).
\end{align}
Furthermore, assuming that $|\rho|^{-1}\omega\delta\rightarrow\infty$ (as $\omega\rightarrow 0$, $\delta\rightarrow 0$, and $|\rho|\rightarrow 0$), then it holds
\begin{align}\label{bl2}
\left\|\nabla u^s\right\|_{L^2(\mathbb{R}^2\setminus\overline{D_\delta})}\rightarrow\infty.
\end{align}
\end{enumerate}
Here, $$\langle\langle 1,\ \tilde{\varphi}_{j_*}\rangle\rangle_{S^f}:=
\int_{-L/2}^{L/2}\ln\frac{(x_1+L/2)^2+\delta^2}{(x_1-L/2)^2+\delta^2}\overline{\tilde\varphi}_{j_*}(x_1)dx_1,$$
and 
$$
\overline{\tilde\varphi}_{j_*}(x_1)=\frac{\tilde\varphi_{j_*}(x_1,1)+\tilde\varphi_{j_*}(x_1,-1)}{2}.
$$
\end{thm}

\begin{proof}
(1) From (\ref{expan01}), it finds
\begin{equation}\label{}
\mathcal {A}_{D_\delta}(\omega)=\mathcal{A}_{D_\delta,0}\left(\mathcal{I}+\mathcal{A}_{D_\delta,0}^{-1}\left(\omega^{2}\ln\omega\mathcal {A}_{D_\delta,1}+\mathcal {O}(\omega^2)\right)\right),
\end{equation}
and then
\begin{equation}\label{}
\psi=\left(\mathcal{I}+\mathcal{A}_{D_\delta,0}^{-1}\left(\omega^{2}\ln\omega\mathcal {A}_{D_\delta,1}+\mathcal {O}(\omega^2)\right)\right)^{-1}\mathcal{A}_{D_\delta,0}^{-1}[f].
\end{equation}
By Lemma \ref{lem4.2} (1), one deduce
\begin{equation}\label{}
\left\|\mathcal{A}_{D_\delta,0}^{-1}\omega^{2}\ln\omega\left(\mathcal {A}_{D_\delta,1}+\mathcal{O}((\ln\omega)^{-1})
\right)\right\|_{\mathcal{L}(\mathcal{H}^*(\partial D_\delta),\mathcal{H}^*(\partial D_\delta))}\lesssim(|\rho|\cdot \mathcal{O}(\delta^{1-\epsilon}))^{-1}\omega^2\ln\omega\left(1+\mathcal{O}((\ln\omega)^{-1})\right).
\end{equation}
Thereby, it holds
\begin{align}
\left\|\psi-\psi_0\right\|=&\left\|\left(\mathcal{I}+\mathcal{A}_{D_\delta,0}^{-1}
\omega^{2}\ln\omega\left(\mathcal
{A}_{1}+\mathcal{O}((\ln\omega)^{-1})
\right)\right)^{-1}\mathcal{A}_{D_\delta,0}^{-1}[f]-\mathcal{A}_{D_\delta,0}^{-1}[f]\right\|\nonumber\\
\lesssim&(|\rho|\cdot \mathcal{O}(\delta^{1-\epsilon}))^{-1}\omega^2\ln\omega\left(1+\mathcal{O}((\ln\omega)^{-1})\right)\left\|\mathcal{A}_{D_\delta,0}^{-1}[f]\right\|\nonumber\\
=&(|\rho|\cdot \mathcal{O}(\delta^{1-\epsilon}))^{-1}\omega^2\ln\omega\left(1+\mathcal{O}((\ln\omega)^{-1})\right)\left\|\psi_0\right\|.\label{4.15}
\end{align}
If $\frac{1}{2}\frac{\theta+1}{\theta-1}\neq\lambda_{j,\delta}$, by Lemma \ref{lem4.2} (2), we get
\begin{align*}
\left\|\psi\right\|\lesssim\left(1+(|\rho|\cdot \mathcal{O}(\delta^{1-\epsilon}))^{-1}\omega^2\ln\omega
\left(1+\mathcal{O}((\ln\omega)^{-1})\right)\right)
\left\|\psi_0\right\|\lesssim(1+c_1)\left\|f\right\|.
\end{align*}
Then, from Lemma \ref{lem4.1}, it yields
\begin{align*}
\left\|\nabla u^s\right\|_{L^2(\mathbb{R}^2\setminus\overline{D_\delta})}^2\lesssim\left\|\psi\right\|^2\lesssim C.
\end{align*}

(2) If $\lambda_{j_*,\delta}=\frac{1}{2}\frac{\theta+1}{\theta-1}-\rho\frac{1}{\theta-1}$, ($j_*\geq1$), then, by using Lemma \ref{lem4.2} (3), we have that
\begin{align}\label{4.16}
\|\psi_{0,c}\|\gtrsim|\rho|^{-1}a_{j_*,\delta}^{-\frac{1}{2}}\left|\langle f,\varphi_{j_*,\delta}\rangle\right|.
\end{align}
Moreover, following a similar estimate as in (\ref{4.15}) and use the results in Lemma \ref{lem4.2} (2) and (3), we see
\begin{align}\label{4.17}
|\rho|^{-1}\omega^2\ln\omega\left(1
+\mathcal{O}((\ln\omega)^{-1})\right)\left\|\psi_0\right\|
\gtrsim\left\|\psi-\psi_0\right\|\gtrsim\left\|\psi_c-\psi_{0,c}\right\|
\gtrsim\left\|\psi_{0,c}\right\|-\left\|\psi_c\right\|.
\end{align}

Combining now (\ref{4.16}) and (\ref{4.17}), and notice that $|\rho|^{-1}\omega^2\ln\omega\left(1
+\mathcal{O}((\ln\omega)^{-1})\right)\leq c_1$ for a sufficiently small $c_1$, we have
\begin{align}\label{}
\left\|\psi_c\right\|\gtrsim\left\|\psi_{0,c}\right\|
-|\rho|^{-1}\omega^2\ln\omega\left(1
+\mathcal{O}((\ln\omega)^{-1})\right)\left\|\psi_0\right\|
\gtrsim|\rho|^{-1}a_{j_*,\delta}^{-\frac{1}{2}}\left|\langle f,\varphi_{j_*,\delta}\rangle\right|.
\end{align}
Thus, we obtain from Lemma \ref{lem4.1} that
\begin{align*}
\left\|\nabla u^s\right\|_{L^2(\mathbb{R}^2\setminus\overline{D_\delta})}^2&\gtrsim
\left\|\psi_c\right\|^2-\omega^2\ln\omega a_{0,\delta}^{-1}\left|\langle \psi,\varphi_{0,\delta}\rangle\right|^2\\
&\gtrsim|\rho|^{-2}a_{j_*,\delta}^{-1}\left|\langle f,\varphi_{j_*,\delta}\rangle\right|^2-\omega^2\ln\omega a_{0,\delta}^{-1}\left|\langle \psi,\varphi_{0,\delta}\rangle\right|^2.
\end{align*}
Furthermore, note that $d_2=0$, by (\ref{eq:eigenasy00}), and applying Lemma \ref{le:app0103}, it follows that
\begin{align*}
\langle f,\varphi_{j_*,\delta}\rangle&=
\int_{\partial D_\delta}\left(-\mathrm{i}\omega(1-\varepsilon_c^{-1})d\cdot\nu(\Bx)+\mathcal{O}(\omega^2\ln\omega)\right)
\left(-\mathcal{S}_{D,0}[\tilde{\varphi}_{j_*}](\tilde{\Bx})+\mathcal{O}(\delta\ln\delta)\right)
d\tilde{\sigma}(\tilde{\Bx})\\
&=\mathrm{i}\omega\delta(1-\varepsilon_c^{-1})\int_{(S^a\cup\iota_{\delta}(P))\cup (S^b\cup\iota_{\delta}(Q)) }d\cdot\nu(\Bx)\mathcal{S}_{D,0}[\tilde{\varphi}_{j_*}](\tilde{\Bx})
d\tilde{\sigma}(\tilde{\Bx})\\
&\ \ \ \ \ \ \ \ \ \ \ \ +\mathrm{i}\omega(1-\varepsilon_c^{-1})\int_{S^f\setminus(\iota_{\delta}(P)\cup \iota_{\delta}(Q))}d\cdot\nu(\Bx)\mathcal{S}_{D,0}[\tilde{\varphi}_{j_*}](\tilde{\Bx})
d\tilde{\sigma}(\tilde{\Bx})+o(\omega\delta)+\mathcal{O}(\omega^2\ln\omega)\\
&=\frac{\mathrm{i}}{2\pi}\omega\delta(1-\varepsilon_c^{-1}) d_1\int_{-L/2}^{L/2}\ln\frac{(x_1+L/2)^2+\delta^2}{(x_1-L/2)^2+\delta^2}\overline{\tilde{\varphi}}_{j_*}(x_1)dx_1 +o(\omega\delta)+\mathcal{O}(\omega^2\ln\omega),
\end{align*}
and
\begin{align*}
a_{j_*,\delta}=-\int_{S^f}\tilde{\varphi}_{j_*}(\tilde{\Bx})\mathcal{S}_{D,0}[\tilde{\varphi}_{j_*}](\tilde{\Bx})
d\tilde{\sigma}(\tilde{\Bx})+\mathcal{O}(\delta\ln\delta).
\end{align*}

The proof is complete.
\end{proof}

\begin{rem}
Note that case (2) in Theorem \ref{thm4.3} is the resonance condition, i.e. $\lambda_{j_*,\delta}=\frac{1}{2}\frac{\theta+1}{\theta-1}-\rho\frac{1}{\theta-1}$. If the lossy parameter of the nanorod $\Im(\varepsilon_c)\rightarrow0$, and when $\lambda_{j_*,\delta}\in(-1/2,\ 1/2)$, it implies $\Re(\varepsilon_c)<0$, which is the negative permittivity materials (see\cite{SPW22,V21}). Furthermore, if $\Im(\varepsilon_c)\rightarrow0$, and $\delta\rightarrow0$, the resonance condition is consistent with $\lambda(\varepsilon_c)-\frac{1}{2}=0$
, which appeared in (\ref{eq:thmanin0101}). In particular, it further deduces that $\Re(\varepsilon_c)\rightarrow0$ and $\lambda_{j_*,\delta}\rightarrow1/2$, which is called the epsilon-near-zero materials (ENZM). Recently, the ENZM have drawn much attention for their intriguing electromagnetic properties (cf. \cite{alam2016,caspani2016}).
\end{rem}


\section*{Acknowledgment}
The work of Y. Deng was supported by NSF grant of China No. 11971487 and NSF grant of Hunan No. 2020JJ2038. The work of H Liu was supported by Hong Kong RGC General Research Funds (project numbers, 11300821, 12301218 and 12302919) and the NSFC-RGC Joint Research Grant (project number, N\_CityU101/21). The work of G. Zheng was supported by NSF grant of China No. 11301168 and NSF grant of Hunan No. 2020JJ4166.


\begin{thebibliography}{99}

%
%
%
%
%
%
%
%
%
%
%
%
%
%
%
%

%
%
%
%

%

%
%
%
%
%
%
%
%
%
%

\bibitem{alam2016} {M. Alam, I. De Leon, R. Boyd}, {\it Large optical nonlinearity of indium tin oxide in its epsilon-near-zero region}, Science, {\bf 352} (2016), 795--797.

\bibitem{ACKLM5}
H. Ammari, G. Ciraolo, H. Kang, H. Lee, and G. W. Milton, {\it Spectral theory of a Neumann
Poincar\'e-type operator and analysis of cloaking due to anomalous localized resonance}, Arch.
Ration. Mech. Anal., {\bf 208} (2013), 667-692.


\bibitem{ACL20} H. Ammari, Y. T. Chow and H. Liu, {\it Quantum ergodicity and localization of plasmon resonances}, {\tt arXiv:2003.03696}

\bibitem{ADM6}
H. Ammari, Y. Deng, and P. Millien, {\it Surface plasmon resonance of nanoparticles and
applications in imaging}, Arch. Ration. Mech. Anal., {\bf 220} (2016), 109-153.


\bibitem{AMRZ14}
H. Ammari, P. Millien, M. Ruiz, and H. Zhang, {\it Mathematical analysis of plasmonic
nanoparticles: the scalar case}, Arch. Ration. Mech. Anal., {\bf 224} (2017), 597-658.


\bibitem{AKL13}
K. Ando, H. Kang and H. Liu, {\it Plasmon resonance with finite frequencies: a validation of the quasi-static approximation for diametrically small inclusions}, SIAM J. Appl. Math., {\bf 76} (2016), 731-749.

\bibitem{BGQ2}
G. Baffou, C. Girard, and R. Quidant, {\it Mapping heat origin in plasmonic structures}, Phys.
Rev. Lett., {\bf 104} (2010), 136805.

\bibitem{BLLW} E. Bl{\aa}sten, H. Li, H. Liu and Y. Wang, {\it Localization and geometrization in plasmon resonances and geometric structures of Neumann-Poincar\'e eigenfunctions}, ESAIM Math. Model. Numer. Anal., \textbf{54} (2020), no. 3, 957--976.

\bibitem{BS11}
G. Bouchitt\'e and B. Schweizer, {\it Cloaking of small objects by anomalous localized resonance},
Quart. J. Mech. Appl. Math., {\bf 63} (2010), 437-463.


\bibitem{caspani2016}
L. Caspani, R. Kaipurath, M. Clerici, et al., {\it Enhanced nonlinear refractive index in epsilon-near-zero materials}, Phys. Rev. Lett. 116 (2016), 233901.


\bibitem{CKKL7}
D. Chung, H. Kang, K. Kim and H. Lee, {\it Cloaking due to anomalous localized resonance
in plasmonic structures of confocal ellipses}, SIAM J. Appl. Math., {\bf 74} (2014),
1691-1707.

\bibitem{D43}
T. V. Duncan, {\it Applications of nanotechnology in food packaging and food
safety: barrier materials, antimicrobials and sensors}, J. Colloid Interface Sci.,
{\bf 363} (1) (2011) 1-24.

\bibitem{DLU7}
Y. Deng, H. Liu and G. Uhlmann, {\it On regularized full- and partial-cloaks in acoustic scattering},
Commun. Part. Differ. Equ., {\bf 42} (6) (2017),  821-851.

\bibitem{DLU72}
Y. Deng, H. Liu and G. Uhlmann, {\it Full and partial cloaking in electromagnetic scattering},  Arch. Ration. Mech. Anal., {\bf 223} (1)(2017), 265-299.
\bibitem{DLZ21}
Y. Deng, H. Liu and G. Zheng, {\it Mathematical analysis of plasmon resonances for curved nanorods}, Journal de Math\'ematiques Pures et Appliqu\'ees, 153 (2021), 248--280.

\bibitem{FDL15}
X. Fang, Y. Deng, J. Li, {\it Plasmon resonance and heat generation in nanostructures},
Mathematical Methods in the Applied Sciences, {\bf 38} (2015), 4663-4672.

\bibitem{fang2021}
X. Fang, Y. Deng and H. Liu, {\it Sharp estimate of electric field from a conductive rod
and application}, Studies in Applied Mathematics, 146 (2021), 279--297.

\bibitem{HDA42}
W. Huynh, J. Dittmer, A. Alivisatos, {\it Hybrid nanorod-polymer solar
cells}, Science, {\bf 295} (5564) (2002) 2425-2427.

\bibitem{JGM40}
N. Jana, L. Gearheart, C. Murphy, {\it Wet chemical synthesis of high
aspect ratio cylindrical gold nanorods}, J. Phys. Chem. B, {\bf 105} (19) (2001)
4065-4067.

\bibitem{JLEE3}
P. Jain, K. Lee, I. El-Sayed, and M. El-Sayed, {\it Calculated absorption and scattering
properties of gold nanoparticles of different size, shape, and composition: Applications in
biomedical imaging and biomedicine}, J. Phys. Chem. B, {\bf110} (2006), 7238-7248.


\bibitem{KLSW8}
R. Kohn, J. Lu, B. Schweizer and M. Weinstein, {\it A variational perspective on cloaking
by anomalous localized resonance}, Comm. Math. Phys., {\bf 328} (2014), 1-27.

\bibitem{LL15}
H. Li and H. Liu,
{\it On anomalous localized resonance and plasmonic cloaking beyond the quasistatic
limit}, Proc. R. Soc. A, {\bf 474}: 20180165.

\bibitem{LLL9}
H. Li, J. Li and H. Liu, {\it On quasi-static cloaking due to anomalous localized resonance in $\mathbb{R}^3$},
SIAM J. Appl. Math., {\bf 75} (2015), 1245-1260.

\bibitem{LLL16}
H. Li, J. Li and H. Liu, {\it On novel elastic structures inducing polariton resonances with finite frequencies
and cloaking due to anomalous localized resonance}, J. Math. Pures Appl., {\bf 120} (2018), 195-219.

\bibitem{LWJC41}
S. Lin, X. Wang, Z. Ji, C. H. Chang, Y. Dong, H. Meng, et al., {\it Aspect ratio
plays a role in the hazard potential of CeO2 nanoparticles in mouse lung and
zebrafish gastrointestinal tract}, ACS Nano, {\bf 8} (5) (2014) 4450-4464.

\bibitem{LSSZ}
H. Liu, Z. Shang, H. Sun and J. Zou, {\it Singular perturbation of reduced wave equation and scattering from an embedded obstacle}, J. Dynam. Differential Equations, \textbf{24} (2012), no. 4, 803--821.

\bibitem{RS} M. Ruiz and O. Schnitzer, {\it Slender-body theory for plasmonic resonance}, Pro. R. Soc. A, \textbf{475}:20190294.

\bibitem{SC1}
D. Sarid and W. Challener, {\it Modern Introduction to Surface Plasmons:
Theory, Mathematical Modeling, and Applications}, Cambridge University Press, New York, 2010.


%
%
%
%

\bibitem{SPW22}
D. Smith, J. Pendry and M. Wiltshire, {\it Metamaterials and negative refractive index}, Science,
{\bf 305} (2004), 788-92.
%
%
%
%

%
%

%
%

%
%
%
%
%
%
%
\bibitem{V21}
V. Veselago, {\it The electrodynamics of substances with simultaneously negative values of $\varepsilon$ and $\mu$ }, Sov. Phys.
Usp., {\bf 10} (1968), 509.

\bibitem{VTBH44}
M. Volder, S. Tawfick, R. Baughman, A. Hart, {\it Carbon nanotubes:
present and future commercial applications}, Science, {\bf 339} (6119)
(2013) 535-539.

\bibitem{WN10}
G. Milton and N. Nicorovici, {\it On the cloaking effects associated with anomalous
localized resonance}, Proc. R. Soc. A, {\bf 462} (2006), 3027-3059.

\bibitem{YA18}
S. Yu and H. Ammari, {\it Plasmonic interaction between nanospheres}. SIAM Review, {\bf 60 (2)} (2018), 356--385.

\bibitem{Z22}
G. Zheng, {\it Mathematical analysis of plasmonic resonance for 2-D photonic crystal}, J. Differ. Equ., {\bf 266} (2019), pp. 5095-5117.



\end{thebibliography}
\end{document}